\documentclass[12pt,final,1p,times]{elsarticle}
\usepackage{amsmath}
\usepackage{xcolor}
\usepackage{multirow}
\usepackage{makecell}
\usepackage{cellspace}
\usepackage{matlab-prettifier}
\usepackage{graphicx, epstopdf}
\usepackage{float}
\DeclareMathOperator*{\es}{\textup{ess sup}}
\allowdisplaybreaks[1]

\journal{ Bulletin des Sciences Mathématiques}
\begin{document}

\newtheorem{theorem}{Theorem}[section]
\newtheorem{lemma}{Lemma}[section]
\newtheorem{remark}{Remark}[section]
\newtheorem{pro}{Proposition}[section]
\newtheorem{example}{Example}[section]
\newdefinition{definition}{Definition}[section]
\newproof{proof}{Proof}
\renewcommand{\theequation}{\thesection.\arabic{equation}}
\newtheorem{thm}{Theorem}[section]
 \newtheorem{cor}[thm]{Corollary}
 \newtheorem{lem}[thm]{Lemma}
 \newtheorem{prop}[thm]{Proposition}
 \newtheorem{defn}[thm]{Definition}
 \newtheorem{rem}[thm]{Remark}
 \newtheorem{ex}{Example}
 \numberwithin{equation}{section}

\begin{frontmatter}

\title{Random average sampling over  quasi shift-invariant spaces on LCA groups}
%\titlerunning{Random average sampling}
\author[Ankush]{Ankush Kumar Garg}
%\ead{ankush16@iisertvm.ac.in}
\author[Arati]{S. Arati}
%\ead{aratishashi@iisertvm.ac.in}
\author[Deva]{P. Devaraj}%\corref{*}}
\ead{devarajp@iisertvm.ac.in}

\cortext[Deva]{Corresponding Author}
\address[]{School of Mathematics, Indian Institute of Science Education and Research Thiruvananthapuram, Vithura, Thiruvananthapuram-695551.}
%Ankush,Arati,Deva

%\maketitle

\begin{abstract}
  The problem of random average sampling and reconstruction over  multiply generated local quasi shift-invariant subspaces of  mixed Lebesgue spaces in the setting of locally compact abelian groups is considered. The sampling inequalities as well as the reconstruction formulae are shown to hold with very high probability if the number of samples is taken to be sufficiently large.
\end{abstract}
%\subjclass[2010]{46E30; 94A20; 94A12

\end{frontmatter}
\section{Introduction}
One of the most dynamically evolving research areas, with both practical and theoretical interests, is that of sampling and reconstruction.
The sampling theorem of Whittaker-Kotel'nikov-Shannon, which is one of the most remarkable results in signal analysis, gives a reconstruction formula for band-limited functions on $\mathbb{R}.$ As the band-limited functions have holomorphic extensions to the whole of $\mathbb{C}$, they are of infinite duration and hence for practical applications, several generalizations of this theorem have been studied in various contexts. An extensive survey of this theory can be seen in \cite{ButzerStens}. Moreover, the measurement of the  exact values of the samples, which depends on the aperture device used, may not be feasible. A more realizable model would be that involving the average samples. Further, the deterministic approach for irregular sampling of functions involving several variables seems to be hard, which led to the consideration of a probabilistic one, namely the random sampling.
\par
Random sampling of multivariate trigonometric polynomials as well as band-limited functions were analysed by Bass and Gröchenig in \cite{bass-grochenig1, bass-grochenig2, bass-grochenig3}.  The problem of random sampling for a shift invariant subspace of $L^p(\mathbb{R}^d)$ generated by a continuous compactly supported function was studied in \cite{yangwei}. A similar problem for functions with bounded derivatives on $(0,1)^d$ was looked into in \cite{YangTao} and that for reproducing kernel subspaces was analysed in \cite{patelsampath, LiSunXian}. Sampling inequalities and reconstruction for certain subsets of multiply generated shift invariant subspaces of Lebesgue spaces and mixed Lebesgue spaces were obtained in \cite{furr, yang} and \cite{jiangli3} respectively. Reconstruction from random average samples of functions in local shift invariant spaces generated by finitely many compactly supported continuous functions on $\mathbb{R}^d$ was studied in \cite{LiWenXian2}. Some other works involving the deterministic average sampling include \cite{sun5, sun6} for band-limited functions, \cite{sun4, AldroubiSunTang, kan, deva1} for various shift-invariant subspaces, \cite{NashedSunXian} for reproducing kernel subspaces and \cite{ sun3, sun2} for spline spaces.

The study of sampling in the group setting has been looked into extensively over the last few decades. For instance, Pensenson, in \cite{Pesenson}, provided a reconstruction formula for spectral entire functions on stratified Lie groups using their samples on discrete subgroups. The error analysis for several iterative methods of reconstruction for band-limited functions on locally compact abelian groups was done in \cite{FeichPandey}. In \cite{FuhrGrochenig}, sampling inequality for closed subspaces, comprising of continuous functions,  of the space of square-integrable functions on locally compact groups with a reproducing kernel was derived using oscillation estimates.  Also, the  sampling inequality for Paley-Wiener spaces on stratified Lie groups was  obtained. Šikić and Wilson \cite{sikic} considered cyclic lattice invariant closed subspaces in the locally compact abelian group setting and provided a reconstruction formula using a periodization condition on the generator. In \cite{Agora}, using a multi-tiling condition, the existence of sampling as well as interpolation sets near the critical density in the case of band-limited functions on LCA groups was proved. In 2017,  Garc\'{i}a et al.  \cite{GarMedVil} studied the theory of sampling for unitary invariant subspaces in the context of LCA groups.

 In this paper, we analyse the problem of random average sampling and reconstruction for signals in  local quasi shift-invariant subspaces of mixed Lebesgue spaces over  locally compact abelian(LCA) groups of  the form $G_{1} \times G_{2},$ where  $G_{1} $ and $ G_{2}$ are LCA groups. The theory of mixed Lebesgue spaces has been elaborated by Benedek and Panzone in \cite{benzone}. The concept of quasi shift-invariance has been considered for analysing sampling and interpolation properties in \cite{Atreas, GroSto, Hamm}. We organize the paper as follows. In Section 2, we discuss the setting of the problem, the assumptions involved and some preliminary lemmas that will be used in the sequel. In Section 3, we prove the sampling inequalities for certain subsets of a local quasi shift-invariant space and in the last section, we  provide reconstruction formulae for the local quasi shift-invariant space under certain conditions.

\section{A  local quasi shift-invariant space } \label{sec2}
\noindent
Let $G_{1}$ and $G_{2}$ be locally compact abelian groups. For $1<p,q<\infty,$ let $L^{p,q}(G_{1} \times G_{2})$ denote the collection of all measurable functions $f:G_{1}\times G_{2}\rightarrow \mathbb{C}$ such that
$$\|f\|_{L^{p,q}(G_{1} \times G_{2})} := \left(\int_{G_{1}}\left(\int_{G_{2}}|f(u,v)|^{q}dv\right)^{\frac{p}{q}}du\right)^{\frac{1}{p}} < \infty,$$
where $du$ and $dv$ denote the Haar measures of $G_{1}$ and $G_{2}$ respectively. The set of all complex valued measurable functions which are essentially bounded is denoted by $L^{\infty}(G_{1} \times G_{2})$ and
\begin{equation*}
\|f\|_{L^{\infty}(G_{1}\times G_{2})}:=  \es_{(u,v)\in (G_{1}\times G_{2})}   |f(u,v)|.
\end{equation*}

 A sequence $\{x_{i}\}$ in a LCA group $G$ is called $U$-separated if there exists a neighbourhood $U$ of the identity element in $G$ such that $(x_{i} +U) \bigcap (x_{j} + U)= \emptyset$ for $i \neq j.$ For a countable index set $I$, let $X=\{x_{s}\}_{s\in I} \subset G_{1}$ and $Y=\{y_{t}\}_{t \in I} \subset G_{2}$ be $U_{1}$-separated and $U_{2}$-separated sequences in $G_{1}$ and $G_{2}$ respectively. Let $\ell^{p,q}$ denote the space of all complex sequences $\{c(s, t):s,t \in I\}$ satisfying
$$\|c\|_{\ell^{p,q}} := \left(\sum_{s}\left(\sum_{t}|c(s, t)|^{q}\right)^{\frac{p}{q}}\right)^{\frac{1}{p}} < \infty.$$
Also, $\ell^{\infty}$ denotes the space of all complex functions $c$ such that $\|c\|_{\ell^{\infty}} := \displaystyle \sup_{s,t}|c(s, t)| < \infty.$\\
\noindent
The following are a few assumptions that we need for our analysis.
\begin{itemize}
\item[$(\text{A}_{1})$] Let $\omega$ be a function in $L^{1}(G_{1} \times G_{2})$ with  support contained in $W=W_{1}\times W_{2},$  where $ W_{1} $  and $W_{2}$ are  compact subsets of   $ G_{1}$ and  $ G_{2}$  containing the corresponding identity elements.

\item[$(\text{A}_{2})$] Let $K = K_1\times K_2$, where $K_1 \subset G_1$ and $K_2 \subset  G_2$ are compact sets with the corresponding  Haar measures of these sets  denoted by $\mu_1$ and $\mu_2.$ Also, let $\widetilde{K}:=K-W.$  Further, let $ \Phi = (\phi_{1}, \phi_{2}, \dots, \phi_{r})^{T},$ where $ \phi_{i} \in L^{p,q}(G_{1} \times G_{2}),  i=1,2, \dots, r$ are compactly supported continuous functions on $G_{1}\times G_{2},$ for which there exist $0<a_{1} \leq a_{2} < \infty$ satisfying
\begin{equation}
a_{1}\|\textbf{c}\|_{\ell^{p,q}} \leq \left\|\sum_{s,t} \textbf{c}(s, t)^{T} \Phi(\cdot-x_{s}, \cdot-y_{t})\right\|_{L^{p,q}(\widetilde{K})} \leq a_{2} \|\textbf{c}\|_{\ell^{p,q}} \label{asmptn1}
\end{equation}
for all $\textbf{c} \in \left(\ell^{p,q} \right)^{r}$  such that  $\textbf{c}(s,t)=0$ whenever $\Phi(\cdot-x_{s}, \cdot-y_{t})\equiv 0$ on $\widetilde{K},$ and $ \displaystyle \| \textbf{c} \|_{\ell^{p,q}} = \sum_{i=1}^{r} \| c_{i}\|_{\ell^{p,q}}.$ In other words,  $\Phi$ has stable shifts.

 \item[$(\text{A}_{3})$] Let $\displaystyle \Omega =  \bigcup_{i=1}^{r} \text{supp}(\phi_{i})$ be such that the compact set $\widetilde{K}- \Omega$ is sequentially compact.

 \item[$(\text{A}_{4})$]  Let $\rho$ be a general probability distribution function over $K$ satisfying
 \begin{equation*}
 0 < {\cal C}_{\rho, 1} \leq \rho(u, v) \leq \mathcal{C}_{\rho, 2} \ \forall \ (u,v) \in K.
 \end{equation*}
\end{itemize}
We wish to emphasize that while considering average sampling, one needs to take into account a domain for the function larger than the sampling domain in order to handle the convolution samples. This has been taken care of in our assumptions above.
\par
Consider the local quasi shift-invariant space $V_{K}(\Phi)$ in $L^{p,q}(G_{1} \times G_{2})$ defined by
\begin{eqnarray*}
&& V_{K}(\Phi):= \bigg\{f\in L^{p,q}(G_{1} \times G_{2}): f= \sum_{s,t} \textbf{c}(s,t)^{T} \Phi(\cdot-x_{s}, \cdot-y_{t}) \ \text{on } \widetilde{K}, \nonumber \\
&& \qquad \qquad \qquad \qquad \qquad \qquad  \qquad \qquad \qquad \qquad \qquad  \text{for some} \ \textbf{c} \in (\ell^{p,q})^{r}\bigg\}.
\end{eqnarray*}
\noindent
The subspace $V_{K}(\Phi)$ is a finite dimensional subspace of $L^{p,q}(\widetilde{K}).$ In fact, there will be only finitely many $(s,t)$ for which $(u - x_{s}, v-y_{t}) \in \Omega ,$ where $u,v \in \widetilde{K}.$ If not, there will be infinitely many $(s,t)$ in the sequentially compact set $\widetilde{K}-\Omega$ which leads to a contradiction to the $U_{1}$-separated and $U_{2}$-separated properties of the sequences $X$ and $Y$ respectively. Let $d$ denote the dimension of $V_{K}(\Phi)$ as a real vector space and $V_{K}^{*}(\Phi)$ denote its normalization, i.e., $V_{K}^{*}(\Phi):= \{f \in V_{K}(\Phi): \|f\|_{L^{p,q}(\widetilde{K})}=1\}.$

\begin{lem} \label{lemCphi}
For any $f \in V_{K}^{*}(\Phi), \|f\|_{L^{\infty}(\widetilde{K})} \leq \dfrac{\widetilde{C}_{\Phi}}{a_{1}},$\\
 where
  \begin{equation*}
 \widetilde{C}_{\Phi}= \sup_{(u,v) \in \widetilde{K}} \sum_{i=1}^{r} \sum_{s,t} |\phi_{i}(u-x_{s}, v-y_{t})|
\end{equation*}
and $a_{1}$ is as in assumption $(\text{A}_{2}).$
\end{lem}
\begin{proof}
For $f\in V_{K}^{*}(\Phi),$ we have
\begin{equation*}
f(u,v)= \sum_{s,t} \textbf{c}(s, t)^{T} \Phi(u-x_{s}, v-y_{t}),  (u,v) \in \widetilde{K},
\end{equation*}
with the convention that  $\textbf{c}(s,t)=0$ whenever $\Phi(\cdot-x_{s}, \cdot-y_{t})\equiv 0$ on $\widetilde{K}.$
Then
 \begin{equation*}
\|f\|_{L^{\infty}(\widetilde{K})} \leq \sup_{(u,v)\in \widetilde{K} } \sum_{i=1}^{r}\sum_{s,t} |c_{i}(s, t)\phi_{i}(u-x_{s}, v-y_{t})| \leq \|\textbf{c}\|_{\ell^{\infty}} \widetilde{C}_{\Phi},
\end{equation*}
where $\displaystyle \|\textbf{c}\|_{\ell^{\infty}} = \sup_{1\leq i \leq r} \|c_{i}\|_{\ell^{\infty}}.$ As $\|\textbf{c}\|_{\ell^{\infty}} \leq \|\textbf{c}\|_{\ell^{p,q}},$ we obtain
$$\|f\|_{L^{\infty}(\widetilde{K})} \leq \|\textbf{c}\|_{\ell^{p,q}} \widetilde{C}_{\Phi} \leq \|f\|_{L^{p,q}(\widetilde{K})} \frac{\widetilde{C}_{\Phi}}{a_{1}} = \frac{\widetilde{C}_{\Phi}}{a_{1}}, \ \text{by using \eqref{asmptn1}}. $$
\hfill{$\Box$}
\end{proof}
\noindent
\begin{rem} By a similar argument, we can prove that $V_{K}(\Phi)\subset L^{\infty}(\widetilde{K}).$
\end{rem}

%Similarly, as in above equation \eqref{Cphi4}, we can define
%\begin{equation*}
% \widetilde{C}_{\Phi}= \sup_{(u,v) \in \widetilde{K}} \sum_{i=1}^{r} \sum_{n,m} |\phi_{i}(u-x_{n}, v-y_{m})|
%\end{equation*}
%and hence for any $f \in V_{K}^{*}(\Phi),$ Lemma \ref{lemCphi} gives
%\begin{equation}
% \|f\|_{L^{\infty}(\widetilde{K})} \leq \dfrac{\widetilde{C}_{\Phi}}{a_{1}}. \label{SupNormKtilde}
%\end{equation}

We shall  make use of  the following lemmas, which are certain variants of the well known Young's inequality,  related to compact subsets of  LCA groups.

\begin{lem}
Let $G$ be a LCA group and $K_{0}, W_{0}$ be two of its compact subsets.   For $1\leq p,q,r \leq \infty$ with $\frac{1}{p}+ \frac{1}{q} = \frac{1}{r}+1,$  suppose $f \in L^{p}(G)$ and $ g \in L^{q}(G)$ such that  $supp(g) \subset W_{0}.$  Then,
\begin{equation}
\|f*g\|_{L^{r}(K_{0})} \leq \|f\|_{L^{p}(\widetilde{K}_{0})} \|g\|_{L^{q}(W_{0})}, \label{young1}
\end{equation}
where $\widetilde{K}_{0}=K_{0}-W_{0}.$ \label{LemYoung1}
\end{lem}

\begin{proof}
Let $1 <p,q,p',q' < \infty $ be such that $\frac{1}{p}+ \frac{1}{p'}=1$ and $\frac{1}{q}+ \frac{1}{q'}=1$.
For $u \in K_{0},$ we have
\begin{eqnarray*}
|(f*g)(u)|  &\leq & \int_{W_{0}}|f(u-v)||g(v)|dv \\
 &=& \int_{W_{0}} \big( |f(u-v)|^{p} |g(v)|^{q}   \big) ^{\frac{1}{r}} |f(u-v)|^{(1- \frac{p}{r})}|g(v)|^{(1- \frac{q}{r})} dv. \\
 \end{eqnarray*}

 As $\frac{1}{p'}+ \frac{1}{q'} + \frac{1}{r}=1,$ we get
  \begin{eqnarray*}
  |(f*g)(u)| & \leq & \left ( \int_{W_{0}} \left|\left( |f(u-v)|^{p} |g(v)|^{q}   \right)^{\frac{1}{r}}\right| ^{r}dv \right )^{\frac{1}{r}} \left( \int_{W_{0}} \left ( |f(u-v)|^{(1- \frac{p}{r})}\right )^{q'}dv \right)^{\frac{1}{q'}} \\
  && \quad \quad \quad \quad  \times \left( \int_{W_{0}}\left(|g(v)|^{(1- \frac{q}{r})} \right )^{p'} dv \right )^{\frac{1}{p'}} ,\\
 \end{eqnarray*}
 using  Hölder's inequality.
 Further, as $q= p'(1- \frac{q}{r})$ and $p= q'(1- \frac{p}{r}),$ it follows that

 \begin{eqnarray*}
 &&|(f*g)(u)| \\
 &&  \leq \left( \int_{W_{0}}  |f(u-v)|^{p} |g(v)|^{q}   dv \right )^{\frac{1}{r}} \left( \int_{W_{0}} |f(u-v)|^{p}dv \right)^{\frac{1}{q'}}  \left( \int_{W_{0}} |g(v)|^{q}dv \right)^{\frac{1}{p'}}\\
 && \leq \left( \int_{W_{0}}  |f(u-v)|^{p} |g(v)|^{q}   dv \right )^{\frac{1}{r}}  \|f\|^{\frac{p}{q'}}_{L^{p}(\widetilde{K}_{0})} \|g\|^{\frac{q}{p'}}_{L^{q}(W_{0})}\\
 &&  \leq  \left( \int_{W_{0}}  |f(u-v)|^{p} |g(v)|^{q}   dv \right )^{\frac{1}{r}}  \|f\|^{(1-\frac{p}{r})}_{L^{p}(\widetilde{K}_{0})} \|g\|^{(1-\frac{q}{r})}_{L^{q}(W_{0})}.
\end{eqnarray*}
Hence,
 \begin{eqnarray*}
 \int_{K_{0}}|(f*g)(u)|^{r}du &\leq & \|f\|^{r-p}_{L^{p}(\widetilde{K}_{0})} \|g\|^{r-q}_{L^{q}(W_{0})} \int_{W_{0}} |g(v)|^{q} \left( \int_{K_{0}}  |f(u-v)|^{p} du   \right )dv \\
& \leq & \|f\|^{r}_{L^{p}(\widetilde{K}_{0})} \|g\|^{r-q}_{L^{q}(W_{0})} \int_{W_{0}} |g(v)|^{q} dv\\
& = & \|f\|^{r}_{L^{p}(\widetilde{K}_{0})} \|g\|^{r}_{L^{q}(W_{0})},
\end{eqnarray*}
and so the inequality \eqref{young1} holds. It is easy to see the proof of \eqref{young1} for the other values of $p$ and $q$.
\hfill{$\Box$}
\end{proof}

\noindent
Now, we prove a version of the Young's inequality for mixed Lebesgue spaces in the group setting.

\begin{lem}
 Let $1<p,q< \infty$, $f \in L^{p,q}(G_{1} \times G_{2})$ and $g \in L^{1}(G_{1} \times G_{2})$ with $ supp (g) \subset
 W_{0} = W_{1} \times W_{2},$   where  $W_{1}$  and $ W_{2}$  are as in assumption (A1).  Then, for compact subsets    $K'_{1}$ and $K'_{2}$ of  $G_{1}$ and $G_{2}$ respectively and  $K_{0} =K'_{1} \times K'_{2},$ we have
 \begin{equation}
\|f*g\|_{L^{p,q}(K_{0})} \leq \|f\|_{L^{p,q}(\widetilde{K}_{0})}\|g\|_{L^{1}(W_{0})}, \label{mixYoung1}
\end{equation}
where $\widetilde{K}_{0}= K_{0} - W_{0}.$
Moreover, if $f\in L^{\infty}(\widetilde{K}_{0}), $ then
\begin{equation}
 \|f*g\|_{L^{\infty}(K_{0})} \leq  \|f\|_{L^{\infty}(\widetilde{K}_{0})}\|g\|_{L^{1}(W_{0})}. \label{mixYoung2}
\end{equation}
   \label{LemYoungMix}
\end{lem}
\begin{proof}
For a fixed $u \in G_{1},$ let  $f_{u}$ and $g_{u}$ be the functions on $G_{2}$ defined by $f_{u}(v):=f(u,v)$ and $g_{u}(v):= g(u,v), v\in G_{2}.$ \\
\noindent
Then,
\begin{eqnarray}
 \|f*g\|^{p}_{L^{p,q}(K_{0})}  &=& \int_{K'_{1}} \left(\int_{K'_{2}}\left|\int_{W_{1}} (f_{u-u'}*g_{u'})(v) du' \right|^{q} dv \right)^{\frac{p}{q}} du \nonumber \\
& = &\int_{K'_{1}} \left \|\int_{W_{1}}(f_{u-u'}* g_{u'})(\cdot)du' \right \|^{p}_{L^{q}(K'_{2} )} du. \label{eq1N22}
\end{eqnarray}
\noindent
Using Minkowski's integral inequality and Lemma \ref{LemYoung1},  we  obtain
\begin{eqnarray*}
 \bigg \|\int_{W_{1}}(f_{u-u'}* g_{u'})(\cdot)du' \bigg \|_{L^{q}(K'_{2})} & \leq & \int_{W_{1}}\|f_{u-u'}* g_{u'} \|_{L^{q}(K'_{2})}du'  \\
& \leq & \int_{W_{1}}\|f_{u-u'}\|_{L^{q}(\widetilde{K}'_{2})} \|g_{u'} \|_{L^{1}(W_{2})}du' .
\end{eqnarray*}
So, from  \eqref{eq1N22}, we  get
\begin{eqnarray*}
 \|f*g\|^{p}_{L^{p,q}(K_{0})} & \leq & \int_{K'_{1}} \left (\int_{W_{1}}\|f_{u-u'}\|_{L^{q}(\widetilde{K}'_{2})} \|g_{u'} \|_{L^{1}(W_{2})}du' \right )^{p} du \\
& = &\int_{K'_{1}} \left (\int_{W_{1}}\widetilde{f}(u-u') \widetilde{g}(u') du' \right )^{p} du ,
\end{eqnarray*}
taking $\widetilde{f}(u)= \|f_{u}\|_{L^{q}(\widetilde{K}'_{2})}$ and $\widetilde{g}(u)= \|g_{u}\|_{L^{1}(W_{2})}, \ u \in G_{1}.$ It is easy to see that $\text{supp}(\widetilde{g}) \subset W_{1}$ and so we have
\begin{equation*}
\|f*g\|^{p}_{L^{p,q}(K_{0})}  \leq  \int_{K'_{1}} |(\widetilde{f}* \widetilde{g})(u)|^{p} du = \|\widetilde{f}* \widetilde{g}\|^{p}_{L^{p}(K'_{1})} .
\end{equation*}

\noindent
Now,
\begin{eqnarray*}
\|\widetilde{f}\|^{p}_{L^{p}(\widetilde{K}'_{1})}
&=&\int_{\widetilde{K}'_{1}} \|f_{u}\|^{p}_{L^{q}(\widetilde{K}'_{2})} du\\ &=& \int_{\widetilde{K}'_{1}} \bigg(\int_{\widetilde{K}'_{2}}|f(u,v)|^{q} dv \bigg )^{\frac{p}{q}} du\\
&=& \|f\|^{p}_{L^{p,q}(\widetilde{K}_{0})} \quad \quad \text{and}
\end{eqnarray*}

\begin{eqnarray*}
\|\widetilde{g}\|^{p}_{L^{1}(W_{1})}
&=& \bigg ( \int_{W_{1}} \|g_{u}\|_{L^{1}(W_{2})} du \bigg )^{p}\\ &=& \bigg( \int_{W_{1}} \int_{W_{2}}|g(u,v)| dv  du \bigg )^{p}\\
&=& \|g\|^{p}_{L^{1}(W_{0})}.
\end{eqnarray*}
Once again using Lemma \ref{LemYoung1}, we get
\begin{equation*}
\|f*g\|^{p}_{L^{p,q}(K_{0})}  \leq  \|\widetilde{f}* \widetilde{g}\|^{p}_{L^{p}(K'_{1})} \leq  \|f\|^{p}_{L^{p,q}(\widetilde{K}_{0})} \|g\|^{p}_{L^{1}(W_{0})},
\end{equation*}
 which proves \eqref{mixYoung1}. We note that it is easy to verify \eqref{mixYoung2}.
 \hfill{$\Box$}
\end{proof}

\bigskip
 Consider the sequence $\{(u_{j}, v_{k})\}_{j,k \in \mathbb{N}}$ of i.i.d. random variables that are drawn from a general probability distribution over a compact set $K$  with $\rho$ being its density function.
 For $f \in V_{K}(\Phi),$
 we define the random variables $Y_{j,k}(f),j,k \in \mathbb{N}$ by
\begin{equation}
Y_{j,k}(f):= |(f*\omega)(u_{j}, v_{k})|- \int_{K} \rho(u,v)|(f*\omega)(u,v)|du dv. \label{RandVar}
\end{equation}

\noindent
Now, we prove some  properties of these random variables in the following lemma.
\begin{lem}
For $f,g \in V_{K}(\Phi),$ the following hold:
\begin{flalign}
 (a) & \text{ The expectation}\  \mathbb{E}[Y_{j,k}(f)]=0. & \label{propert1}\\
 (b)  &\ \|Y_{j,k}(f)- Y_{j,k}(g)\|_{\ell^{\infty}} \leq 2 \|f-g\|_{L^{\infty}(\widetilde{K})} \|\omega\|_{L^{1}(W)}. \label{propert3}\\
(c) &\ Var \big(Y_{j,k}(f)-Y_{j,k}(g) \big ) \leq 4 \|f-g\|^{2}_{L^{\infty}(\widetilde{K})} \|\omega\|^{2}_{L^{1}(W)} . \label{Prop5}
 \end{flalign}
 Furthermore, if  $f \in V_{K}^{*}(\Phi),$ then
 \begin{flalign}
 (d) & \ \|Y_{j,k}(f)\|_{\ell^{\infty}} \leq \frac{\widetilde{C}_{\Phi}}{a_{1}} \|\omega\|_{L^{1}(W)} \mbox{ and } &  \label{property*} \\
 (e) &\ Var(Y_{j,k}(f)) \leq \left(\frac{\widetilde{C}_{\Phi}}{a_{1}}\right)^{2} \|\omega\|^{2}_{L^{1}(W)},& \label{property**}
  \end{flalign}
  where  $\widetilde{C}_{\Phi}$ and  $a_{1}$ are as in Lemma  \ref{lemCphi}.
\end{lem}
\begin{proof}
We observe that (a) is trivial.

\vskip 1em
\noindent (b)  Consider,
\begin{eqnarray*}
&& \|Y_{j,k}(f)- Y_{j,k}(g)\|_{\ell^{\infty}}\\
&& \quad \leq  \sup_{j,k}\left| |((f-g)\ast \omega)(u_{j}, v_{k})| +   \int_{K} \rho(u,v)|((f-g)*\omega)(u,v) | dudv \right|\\&&
 \quad \leq  \sup_{(u',v') \in K}\left( | ((f-g) \ast \omega )(u', v')| + \|(f-g) \ast \omega\|_{L^{\infty}(K)}  \int_{K} \rho(u,v)dudv \right )\\
&& \quad \leq 2\|(f-g) \ast \omega\|_{L^{\infty}(K)}\\
&& \quad \leq 2\|f -g\|_{L^{\infty}(\widetilde{K})} \|\omega\|_{L^{1}(W)},
\end{eqnarray*}
by \eqref{mixYoung2}.

\vskip 1em
\noindent (c) In view of  \eqref{propert1}, we have
 \begin{eqnarray*} && Var \big(Y_{j,k}(f)-Y_{j,k}(g) \big ) \\
 && \quad  = \mathbb{E}\big[\big(Y_{j,k}(f)-Y_{j,k}(g) \big )^{2}\big] \\
 && \quad \leq \mathbb{E}\left [\bigg(\left|\big( (f-g) \ast \omega \big )(u_{j}, v_{k}) \right| +  \int_{K}\rho(u,v) \left|\big((f-g)*\omega \big)(u,v)\right | dudv \bigg)^{2} \right] \\
  && \quad = \mathbb{E}\big[|\big( (f-g) \ast \omega \big )(u_{j}, v_{k})|^{2}\big] +\bigg(  \int_{K}\rho(u,v) |\big((f-g)*\omega \big)(u,v)| dudv\bigg)^{2}\\
 && \quad \quad \quad  + \ 2 \bigg(  \int_{K}\rho(u,v) |\big((f-g)*\omega \big)(u,v)| dudv\bigg)\mathbb{E}\big[|\big( (f-g) \ast \omega \big )(u_{j}, v_{k})|\big]\\
&&  \quad \leq 4\|(f-g) \ast \omega\|^{2}_{L^{\infty}(K)}\\
&& \quad \leq 4  \|f-g\|^{2}_{L^{\infty}(\widetilde{K})} \|\omega\|^{2}_{L^{1}(W)},
\end{eqnarray*}
by \eqref{mixYoung2}.

\vskip 1em
\noindent (d) As  $ \displaystyle \int_{K} \rho(u,v)|(f*\omega)(u,v)|dudv \leq \|f*\omega\|_{L^{ \infty}(K)},$ we have
 \begin{eqnarray*}
 \|Y_{j,k}(f)\|_{\ell^{\infty}} & =&  \sup_{j,k}\bigg| \left|(f \ast \omega)(u_{j}, v_{k})\right| -  \int_{K} \rho(u,v)|(f*\omega)(u,v)|dudv  \bigg|\\
 & \leq & \sup_{j,k} \max\bigg\{ \left|(f \ast \omega)(u_{j}, v_{k}) \right| ,  \int_{K} \rho(u,v)|(f*\omega)(u,v)|dudv \bigg\}\\
& \leq & \|f \ast \omega\|_{L^{\infty}(K)}\\
& \leq & \|f \|_{L^{\infty}(\widetilde{K})} \|\omega\|_{L^{1}(W)} \\
& \leq & \frac{\widetilde{C}_{\Phi}}{a_{1}} \|\omega\|_{L^{1}(W)},
\end{eqnarray*}
using \eqref{mixYoung2} and Lemma \ref{lemCphi}.

\vskip 1em
\noindent (e) Using \eqref{propert1}, we get $Var(Y_{j,k}(f))= \mathbb{E}[Y_{j,k}(f)^{2}] $
and so, \begin{eqnarray*} \mathbb{E}[(Y_{j,k}(f))^{2}] &=& \mathbb{E}\left[|(f \ast \omega)(u_{j}, v_{k})|^{2} \right] + \mathbb{E}\left[\bigg( \int_{K} \rho(u,v)|(f*\omega)(u,v)|dudv \bigg )^{2} \right ] \\
&& \quad  \quad - \ 2 \mathbb{E}\left[|(f \ast \omega)(u_{j}, v_{k})| \bigg ( \int_{K} \rho(u,v)|(f*\omega)(u,v)|dudv \bigg ) \right ] \\
&  = & \mathbb{E}\left[|(f \ast \omega)(u_{j}, v_{k})|^{2} \right]- \bigg( \int_{K} \rho(u,v)|(f*\omega)(u,v)|dudv \bigg)^{2},
\end{eqnarray*}
 which leads to the inequality
\begin{eqnarray*}
Var(Y_{j,k}(f)) &\leq& \ \mathbb{E}[|(f \ast \omega)(u_{j}, v_{k})|^{2}]\\
&\leq& \|f \ast \omega\|^{2}_{L^{\infty}(K)}\\
&\leq&   \|f\|^{2}_{L^{\infty}(\widetilde{K})} \|\omega\|^{2}_{L^{1}(W)}\\
&\leq&   \left(\frac{\widetilde{C}_{\Phi}}{a_{1}}\right)^{2} \|\omega\|^{2}_{L^{1}(W)},
\end{eqnarray*}
using \eqref{mixYoung2} and Lemma \ref{lemCphi}.
\hfill{$\Box$}
\end{proof}

We note that $V^{*}_{K}(\Phi)$ is a compact  subset of $V_{K}(\Phi)$ equipped  with the $L^{\infty}(\widetilde{K})$-norm,   which is indeed a finite dimensional Banach space. So, $V^{*}_{K}(\Phi)$ can be covered by a finite number of closed balls in  $V_{K}(\Phi)$ of a given radius $\eta>0.$  In the following lemma, we provide an upper bound for the least number of such balls. In other words,  an estimate for  its covering number is given.

\begin{lem}  \label{LemCovering}
For any $\eta>0,$  a bound for the covering number $\mathcal{N}(V_{K}^{*}(\Phi), \eta)$ of $V_{K}^{*}(\Phi)$ with respect to the norm $\|\cdot\|_{L^{\infty}(\widetilde{K})}$ is given by
\begin{equation}
\mathcal{N}(V_{K}^{*}(\Phi), \eta) \leq \exp \left(d \ln \left(\frac{4\widetilde{C}_{\phi}}{\eta a_{1}}\right)\right), \label{covrngEq}
\end{equation}
where $\widetilde{C}_{\phi}$ and $a_{1}$ are as in Lemma \ref{lemCphi} and $d$ is the dimension of $V_{K}(\Phi).$
\end{lem}

\begin{proof}

By Lemma  \ref{lemCphi}, $V^{*}_{K}(\Phi)$ is contained in the closed ball $B\left(0, \tfrac{\widetilde{C}_{\Phi}}{a_{1}}\right)$ of radius
 $\tfrac{\widetilde{C}_{\Phi}}{a_{1}}$
centered at the origin in the  Banach space $ \left(V_{K}(\Phi),\|\cdot\|_{L^\infty(\widetilde{K}} \right )$. For $\eta>0,$ the covering number  $\mathcal{N}\left(B\left(0, \tfrac{\widetilde{C}_{\Phi}}{a_{1}}\right), \eta\right)$ of the ball $B\left(0, \tfrac{\widetilde{C}_{\Phi}}{a_{1}}\right)$ is bounded by  $ \exp \left(d \ln \left(\frac{4\widetilde{C}_{\phi}}{\eta a_{1}}\right)\right),$ by Proposition 5 in
\cite{CukerSmale}. Clearly, $\mathcal{N}(V^{*}_{K}(\Phi),\eta)\le
\mathcal{N}\left(B\left(0, \tfrac{\widetilde{C}_{\Phi}}{a_{1}}\right), \eta\right)$ and hence   (\ref{covrngEq}) follows.
\hfill{$\Box$}
\end{proof}

\section{Sampling inequalities} \label{sec3}
In this section, we prove  sampling inequalities  for two subsets of  $V_{K}(\Phi),$ namely  $V^{p,q}_{K,\alpha}(\Phi, \theta)$  and $V_{K}(\Phi,\omega,\mu),$  which are defined as follows:

\noindent
For  $\theta,\,\alpha>0,$
$$V^{p,q}_{K,\alpha}(\Phi, \theta):=\left\{f\in  V_{K}(\Phi): \|f*\omega\|_{L^{p,q}(K)} \geq \theta\text{ and }\|f\|_{L^{p,q}(\widetilde{K})}\leq\alpha\right\} $$
\noindent  and
for $ 0<\mu\leq 1,$
  $$V_{K}(\Phi,\omega,\mu):=\left\{f\in V_K(\Phi):\mu\|\omega\|_{L^1(W)}\|f\|_{L^{p,q}(\widetilde{K})}\leq\|f\ast\omega\|_{L^{1}(K)}\right\}.$$

The proofs of these sampling inequalities will make use of an important lemma,  which is stated below.

\begin{lem} Let $ \omega$ and $\Phi$ satisfy the assumptions $(\text{A}_{1})- (\text{A}_{3})$  and $Y_{j,k}$ be as defined in \eqref{RandVar}. Then for any $m, n \in \mathbb{N}$ and \\
$\mathfrak{p} >  108  \sqrt{2} (\ln 2) d \left(1+ \left(1+\dfrac{3nm}{2  \sqrt{2} (\ln 2)d } \right )^{\frac{1}{2}} \right ) \|\omega\|_{L^{1}(W)},$
 the following inequality holds:
 \begin{eqnarray}
 && Prob\left(\sup_{f \in V_{K}^{*}(\Phi)} \left|\sum_{j=1}^{n} \sum_{k=1}^{m} Y_{j,k}(f)\right|  \geq \mathfrak{p} \right) \nonumber \\
    && \quad < \mathcal{A}_{1} \exp \left(  \dfrac{-3a_{1}^{2}\mathfrak{p}^{2}}{4\widetilde{C}_{\Phi}\|\omega\|_{L^{1}(W)}(6 n m\widetilde{C}_{\Phi} \|\omega\|_{L^{1}(W)}+  \mathfrak{p} a_{1})}\right) \label{3eq10} \\
&&      \quad \quad + \mathcal{A}_{2} \exp \left(  \dfrac{- \mathfrak{p}^{2}}{72 \sqrt{2}\|\omega\|_{L^{1}(W)}(81 n m \|\omega\|_{L^{1}(W)}+  \mathfrak{p})}\right), \nonumber \end{eqnarray}
where $\mathcal{A}_{1}= 2 \exp \left(d \ln \left(\dfrac{8\widetilde{C}_{\Phi}}{a_{1}}\right) \right), \mathcal{A}_{2}= \dfrac{4}{3 d(\ln 2)^{2} } \left(\dfrac{4 \widetilde{C}_{\Phi}}{a_{1}}\right)^{2d}$ and,  $ \widetilde{C}_{\Phi}$ and $a_{1}$ are as in Lemma \ref{lemCphi}.
  \label{lemm3.3}
\end{lem}

\begin{proof}
 Let $l \in \mathbb{N}$ and $n_{l}$  denote the covering number of $V_{K}^{*}(\Phi)$ with respect to the norm $\|\cdot\|_{L^{\infty}(\widetilde{K})}$ and radius  $2^{-l}.$  We consider a cover for $V_{K}^{*}(\Phi)$ consisting of  $n_{l}$ balls $B(h_{l,i},2^{-l})$ in $V_{K}(\Phi)$ having radius $2^{-l}$ and centered at $h_{l,i}, 1\le i\le n_{l}.$  Without any loss of generality, we may assume that each of these balls have non-empty intersection with $V_{K}^{*}(\Phi).$ Let $g_{l,i}\in V_{K}^{*}(\Phi)$  be chosen such that it is closest to $h_{l,i}$ for each $i.$  Let $\mathcal{E}_{l}:=\{g_{l,i}:1\le i\le n_{l}\}.$ For  $f \in V_{K}^{*}(\Phi),$ there exists $f_{l} \in \mathcal{E}_{l}$ such that $\|f-f_{l}\|_{L^{\infty}(\widetilde{K}) } \le 2^{1-l}.$  In view of \eqref{propert3}, we can see that for  $ j,k\in   \mathbb{N},$
 \begin{equation}
 Y_{j,k}(f)=Y_{j,k}(f_{1}) + (Y_{j,k}(f_{2}) - Y_{j,k}(f_{1})) + (Y_{j,k}(f_{3})-Y_{j,k}(f_{2}))+ \cdots. \label{equation*}
\end{equation}
 Now, for $\mathfrak{p}\ge 0,$ we define a sequence of events  $\{\textbf{E}_{l}\}_{l \in \mathbb{N}}$ as follows:
 \begin{equation*}
 \textbf{E}_{1}= \left\{\text{there exists }  g \in \mathcal{E}_{1} \ \text{such that} \ \left|\sum_{j=1}^{n} \sum_{k=1}^{m} Y_{j,k}(g)\right| \geq \dfrac{\mathfrak{p}}{2}\right\}
\end{equation*}
and for $l \geq 2$, \\

$\textbf{E}_{l}= \bigg\{\text{there exist }  g \in \mathcal{E}_{l}\   \text{and}\  h \in \mathcal{E}_{l-1}\ \text{with} \ \|g- h\|_{L^{\infty}(\widetilde{K})}\leq 6\cdot2^{-l}\\ \text{\ \ \ \ \ \ \ \ \ \ \ \ \ \ \ \ \ \ \ \ \ \ \ \ \ \ \ \ \ \ \ \ \ \ \ \ \ \ such that}  \ \bigg| \displaystyle {\sum_{j=1}^{n} \sum_{k=1}^{m} \left(Y_{j,k}(g)- Y_{j,k}(h) \right)\bigg| \geq \dfrac{\mathfrak{p}}{2l^{2}}} \bigg\}.$

 \noindent Suppose $\displaystyle{ \sup_{f \in V_{K}^{*}(\Phi)} \left|\sum_{j=1}^{n} \sum_{k=1}^{m} Y_{j,k}(f)\right|> \mathfrak{p}}$ is true.  Then, there exists some $l \in \mathbb{N}$ for which  the event $\mathbf{E}_{l}$ will hold. If  this  is not the case, then for $f \in V_{K}^{*}(\Phi),$   \eqref{equation*} gives
 \begin{eqnarray*}
 \left|\sum_{j=1}^{n} \sum_{k=1}^{m} Y_{j,k}(f) \right| & \leq & \left|\sum_{j=1}^{n} \sum_{k=1}^{m} Y_{j,k}(f_{1})\right| + \sum_{l=2}^{\infty} \left|\sum_{j=1}^{n} \sum_{k=1}^{m} \big(Y_{j,k}(f_{l})- Y_{j,k}(f_{l-1}) \big ) \right | \\
 & < &\dfrac{\mathfrak{p}}{2} + \sum_{l=2}^{\infty} \dfrac{\mathfrak{p}}{2l^{2}}=\dfrac{\pi^{2} \mathfrak{p}}{12} < \mathfrak{p},
\end{eqnarray*}
which is a contradiction.
Now, we shall estimate an upper bound for the probability of each of the events $\textbf{E}_{l}, l \in \mathbb{N}.$ As the cardinality  of $\mathcal{E}_{l}$ is $n_{l}$, we get

%of we have $\mathcal{N}(V_{K}^{*}(\Phi), 2^{-l})$ balls of radius $2^{-l},$ thus in presence of  Lemma \ref{LemCovering} and Lemma \ref{LemProb}, the probability for the event $\textbf{E}_{1}$ is estimated by

\begin{eqnarray*}
 Prob(\textbf{E}_{1})  & \leq & n_1 \times Prob\left(\left|\sum_{j=1}^{n} \sum_{k=1}^{m} Y_{j,k}(g)\right |\geq \frac{\mathfrak{p}}{2} : g \in \mathcal{E}_{1}\right)\\
 & \leq & \exp \left(d \ln \left(\frac{8 \widetilde{C}_{\Phi}}{a_{1}}\right)\right) Prob\left(\left|\sum_{j=1}^{n} \sum_{k=1}^{m} Y_{j,k}(g)\right |\geq \frac{\mathfrak{p}}{2} : g\in \mathcal{E}_{1}\right),
%& \leq &\exp \left(d \ln \left(\frac{8 \widetilde{C}_{\Phi}}{a_{1}}\right)\right) \nonumber \\
%&& \quad \quad \quad \quad \times 2 \exp \left(- \dfrac{(\frac{\mathfrak{p}}{2})^{2}}{2nm \|Var(Y_{j,k}(f))\|_{\ell^{ \infty}} +\frac{2}{3} \|Y_{j,k}(f)\|_{\ell^{\infty}} \frac{\mathfrak{p}}{2}}  \right ).
\end{eqnarray*}
by Lemma  \ref{LemCovering}. Using  Bernstein's inequality given in   \cite{bennett},  we may write
 for $ g \in \mathcal{E}_{1},$
\begin{eqnarray*}
Prob\left(\left|\sum_{j=1}^{n} \sum_{k=1}^{m} Y_{j,k}(g)\right |\geq \frac{\mathfrak{p}}{2} \right)<
2  \exp \left( \frac{-(\frac{\mathfrak{p}}{2})^{2}}{2  \sum\limits_{j=1}^{n} \sum\limits_{k=1}^{m}  Var(Y_{j,k}(g)) +\frac{2}{3} M \frac{\mathfrak{p}}{2}}  \right ),
\end{eqnarray*}
 where  $M$ satisfies  $\left|Y_{j,k}(g)-E[Y_{j,k}(g)] \right|\le M $ for all $j,k.$ By  \eqref{propert1} and  \eqref{property*},
we may choose
$M=\frac{\widetilde{C}_{\Phi}}{a_{1}} \|\omega\|_{L^{1}(W)}$. Further, making use of  \eqref{property**}, it follows that
\begin{eqnarray}
 Prob(\textbf{E}_{1}) & < & 2\exp \left(d \ln \left(\frac{8 \widetilde{C}_{\Phi}}{a_{1}}\right)\right) \nonumber \\
&& \times \exp \left( \dfrac{- 3a_{1}^{2}\mathfrak{p}^{2}}{4\widetilde{C}_{\Phi}\|\omega\|_{L^{1}(W)} \left(6 n m\widetilde{C}_{\Phi} \|\omega\|_{L^{1}(W)}+  \mathfrak{p} a_{1} \right)}  \right ). \nonumber \\ \label{ProbE1}
\end{eqnarray}

\noindent
For the events $\textbf{E}_{l}, l \geq 2,$ we have
\begin{eqnarray}
&& Prob(\textbf{E}_{l}) \leq    n_{l} n_{l-1}
Prob \bigg(\bigg|\sum_{j=1}^{n} \sum_{k=1}^{m} \left(Y_{j,k}(g)- Y_{j,k}(h) \right) \bigg | \geq \frac{\mathfrak{p}}{2l^{2}} : g \in \mathcal{E}_{l},  \nonumber \\
&& \quad \quad \quad \quad \quad \quad \quad \quad \quad \quad \quad \quad \quad \quad  \quad  h \in \mathcal{E}_{l-1}, \|g-h\|_{L^{\infty}(\widetilde{K})} \leq 6\cdot2^{-l} \bigg).\nonumber \\ \label{ProbEl}
\end{eqnarray}
By using  Lemma \ref{LemCovering}, we obtain
\begin{eqnarray}
  n_{l} n_{l-1}
&& \leq \exp \left(d \ln \left(\frac{4 \widetilde{C}_{\Phi}}{2^{-l}a_{1}}\right)\right) \times \exp \left(d \ln \left(\frac{4 \widetilde{C}_{\Phi}}{2^{-l+1}a_{1}}\right)\right) \nonumber \\
&& = \exp \left(d \left(\ln \left(\frac{4 \widetilde{C}_{\Phi}}{a_{1}}\right)^{2} + \ln \left(\frac{1}{2^{-l} 2^{-l+1}} \right) \right ) \right ) \nonumber \\
&& \leq \left( \frac{4 \widetilde{C}_{\Phi}}{a_{1}}\right)^{2d} \exp (2 l d \ln 2). \label{LemEq1}
\end{eqnarray}
In view of \eqref{propert3} and \eqref{Prop5}, we obtain
\begin{eqnarray}
\|Y_{j,k}(g)- Y_{j,k}(h)\|_{\ell^{\infty}} \leq 12 \cdot 2^{-l}  \|\omega\|_{L^{1}(W)} \ \ \text{and} \nonumber \\
\|Var \left(Y_{j,k}(g)- Y_{j,k}(h)\right)\|_{\ell^{\infty}} \leq 144 \cdot 2^{-2l} \|\omega\|^{2}_{L^{1}(W)}, \label{LemEq}
\end{eqnarray}
 for  $g \in \mathcal{E}_{l}$ and $h \in \mathcal{E}_{l-1}$ satisfying
$\|g-h\|_{L^{\infty}(\widetilde{K})} \leq 6 \cdot 2^{-l},$ $l \geq 2.$
For such $g$ and $h,$ once again applying the Bernstein's inequality and making use of the estimates obtained in \eqref{LemEq},  we get
\begin{eqnarray}
 && Prob \left(\left|\sum_{j=1}^{n} \sum_{k=1}^{m} \bigg(Y_{j,k}(g)- Y_{j,k}(h) \bigg) \right | \geq \frac{\mathfrak{p}}{2l^{2}} \right ) \nonumber \\
% && < 2 \exp \left( \dfrac{-(\frac{\mathfrak{p}}{2 l^{2}})^{2}}{2nm \|Var \left(Y_{j,k}(g)- Y_{j,k}(h)\right)\|_{\ell^{\infty}} %+\frac{2}{3} \|Y_{j,k}(g)- Y_{j,k}(h)\|_{\ell^{\infty}} \frac{\mathfrak{p}}{2l^{2}}}  \right ) \nonumber \\
 && < 2 \exp \left( \dfrac{-(\frac{\mathfrak{p}}{2 l^{2}})^{2}}{288 nm \cdot 2^{-2l} \|\omega\|^{2}_{L^{1}(W)} + 8  \cdot 2^{-l}  \|\omega\|_{L^{1}(W)} \frac{\mathfrak{p}}{2l^{2}}}  \right ) \nonumber \\
 && = 2 \exp \left( \dfrac{- 2^{l} \mathfrak{p}^{2}}{ 4l^{2} \left(288 nm  (l^{2} 2^{-l} ) \|\omega\|^{2}_{L^{1}(W)} + 4 \mathfrak{p} \|\omega\|_{L^{1}(W)}  \right)}  \right ) . \nonumber
\end{eqnarray}
As  $l^{2}2^{-l}\le\dfrac{9}{8}$  for $l \geq 2$, we get
\begin{eqnarray}
&& Prob \left (\left|\sum_{j=1}^{n} \sum_{k=1}^{m} \bigg(Y_{j,k}(g)- Y_{j,k}(h) \bigg) \right | \geq \frac{\mathfrak{p}}{2l^{2}} \right ) \nonumber \\
&& < 2 \exp \left( \dfrac{ - 2^{l} \mathfrak{p}^{2}}{ 16 l^{2} \left(81 nm   \|\omega\|^{2}_{L^{1}(W)} +  \mathfrak{p} \|\omega\|_{L^{1}(W)}  \right)}  \right ) \nonumber \\
&& = 2 \exp \left(- \frac{\nu 2^{l}}{l^{2}} \right), \label{eq23}
\end{eqnarray}
where $\nu = \dfrac{\mathfrak{p}^{2}}{ 16 \left(81nm   \|\omega\|^{2}_{L^{1}(W)} + \mathfrak{p} \|\omega\|_{L^{1}(W)}  \right)}. $\\
Furthermore, \eqref{ProbEl}, \eqref{LemEq1} and \eqref{eq23} together  give

\begin{eqnarray}
 Prob \left(\bigcup_{l=2}^{\infty}\textbf{E}_{l} \right) &< & 2 \left( \frac{4 \widetilde{C}_{\Phi}}{a_{1}}\right)^{2d} \sum_{l=2}^{\infty} \exp \left(2 l d \ln 2 - \frac{\nu 2^{l}}{l^{2}} \right) \nonumber \\
   & = & 2 \left( \frac{4 \widetilde{C}_{\Phi}}{a_{1}}\right)^{2d} \sum_{l=2}^{\infty} \exp \left( - \nu 2^{\frac{l}{2}}\left(\frac{2^{\frac{l}{2}}}{l^{2}}  - \frac{2ld\ln 2}{\nu 2^{\frac{l}{2}}} \right ) \right). \nonumber  \\ \label{eq24}
\end{eqnarray}
\noindent
It can be observed that $\displaystyle{ \min_{l \geq 2} \dfrac{2^{\frac{l}{2}}}{l^{2}}}= \dfrac{2}{9}$
 and
 $\displaystyle \max_{l \geq 2} \dfrac{l}{2^{\frac{l}{2}}}= \dfrac{3}{\sqrt{8}}$.
Thus, we obtain
\begin{equation*}
 \dfrac{2^{\frac{l}{2}}}{l^{2}} - \dfrac{2ld\ln 2 }{\nu 2^{\frac{l}{2}}} \nonumber  \geq  \dfrac{2}{9} - \dfrac{24  \sqrt{2}  d (\ln 2)  \left(81nm\|\omega\|^{2}_{L^{1}(W)}+\mathfrak{p} \|\omega\|_{L^{1}(W)} \right)}{\mathfrak{p}^{2}}.
\end{equation*}
Now, the assumption made on $\mathfrak{p}$ in the hypothesis gives
\begin{equation*}
 \mathfrak{p}^{2} - 2\left(108 \sqrt{2}  d(\ln 2)   \|\omega\|_{L^{1}(W)}\right) \mathfrak{p}  - 216 \times 81  \sqrt{2} d(\ln 2)   nm \|\omega\|^{2}_{L^{1}(W)}  > 0,
\end{equation*}
which implies that
\begin{equation}
\dfrac{2}{9} - \dfrac{24  \sqrt{2}  d (\ln 2)  \left(81nm\|\omega\|^{2}_{L^{1}(W)}+\mathfrak{p} \|\omega\|_{L^{1}(W)} \right)}{\mathfrak{p}^{2}} > \frac{1}{9}. \label{eq24.5}
\end{equation}
Hence, the above inequality together with \eqref{eq24} give
\begin{equation}
Prob \left(\bigcup_{l=2}^{\infty}\textbf{E}_{l} \right) <  2 \left( \frac{4 \widetilde{C}_{\Phi}}{a_{1}}\right)^{2d} \sum_{l=2}^{\infty} \exp \left( - \frac{\nu 2^{\frac{l}{2}}}{9} \right). \label{eq25}
\end{equation}
By making use of the  integral test,  we may write
\begin{equation*}
\sum_{l=2}^{\infty} \exp \left( - \frac{\nu 2^{\frac{l}{2}}}{9} \right) \leq \int_{1}^{\infty} e^{\tfrac{-\nu 2^{\frac{x}{2}}}{9}} dx \leq \dfrac{1}{\sqrt{2} \ln \sqrt{2}} \int_{\sqrt{2}}^{\infty} e^{\frac{-\nu t}{9}} dt = \dfrac{9 e^{\frac{-\sqrt{2} \nu }{9}}}{\sqrt{2} \nu  \ln \sqrt{2}}.
\end{equation*}
Also, from inequality \eqref{eq24.5}, we  obtain
\begin{equation*}
\nu = \dfrac{\mathfrak{p}^{2}}{ 16 \left(81nm   \|\omega\|^{2}_{L^{1}(W)} + \mathfrak{p} \|\omega\|_{L^{1}(W)}  \right)} > 27 \sqrt{2} d(\ln \sqrt{2}) ,
\end{equation*}
which further implies that
\begin{eqnarray*}
 \sum_{l=2}^{\infty} \exp \left( - \frac{\nu 2^{\frac{l}{2}}}{9} \right) &< & \dfrac{9 e^{\frac{-\sqrt{2} \nu }{9}}}{\sqrt{2}  (\ln \sqrt{2})\left(27 \sqrt{2} d(\ln \sqrt{2})  \right)}   \\
&=& \frac{2}{3 d(\ln 2)^{2} } \exp \left(  \dfrac{ - \mathfrak{p}^{2}}{72 \sqrt{2}\|\omega\|_{L^{1}(W)}\left(81 n m \|\omega\|_{L^{1}(W)}+  \mathfrak{p} \right)}\right).
\end{eqnarray*}
So, from this inequality  and \eqref{eq25}, it follows that
\begin{eqnarray*}
 Prob \left(\bigcup_{l=2}^{\infty}\textbf{E}_{l} \right) &< &  \dfrac{4}{3d(\ln 2)^{2} } \left(\frac{4 \widetilde{C}_{\Phi}}{a_{1}}\right)^{2d} \\
 && \times \exp \left( \dfrac{ - \mathfrak{p}^{2}}{72 \sqrt{2}\|\omega\|_{L^{1}(W)}(81 n m \|\omega\|_{L^{1}(W)}+  \mathfrak{p})}\right).
\end{eqnarray*}
Therefore, \eqref{ProbE1} and the above inequality together  prove \eqref{3eq10}.
\hfill{$\Box$}
\end{proof}

Now, we prove the sampling inequality for  the functions in $V^{p,q}_{K,\alpha}(\Phi,\theta).$ Towards this end, we consider the following subsets. For  $\zeta>0,$
$$V^{p,q}_{K}(\Phi, \zeta):=\left\{f\in  V_{K}(\Phi): \|f*\omega\|_{L^{p,q}(K)} \geq \zeta\right\} $$ and we denote the set of all normalized functions in this set by $V^{p,q,*}_{K}(\Phi, \zeta).$ In other words, $V^{p,q,*}_{K}(\Phi, \zeta)=V^{p,q}_{K}(\Phi, \zeta)\bigcap V^{*}_{K}(\Phi).$

\begin{thm}\label{Sampl_Ineq_1}
Let $ \omega, \Phi$ and $\rho$  satisfy the assumptions $(\text{A}_{1})- (\text{A}_{4})$ and $\{(u_{j}, v_{k})\}_{j,k \in \mathbb{N}}$  be a  sequence of i.i.d. random variables over $K$ with probability density function $\rho.$ Then for any $\gamma \in (0,1)$ and
\begin{eqnarray}
nm &>& \frac{108 \sqrt{2} (\ln 2) d \|\omega\|_{L^{1}(W)}\alpha^{2pq}}{\left(\gamma \mathcal{C}_{\rho, 1} \mu_1^{1-q} \mu_{2}^{1-p} \left(\frac{\widetilde{C}_{\Phi}}{a_{1}}\|\omega\|_{L^{1}(W)}\right)^{1-pq} \theta^{pq}\right)^{2}} \nonumber \\
&& \times \left(162 \|\omega\|_{L^{1}(W)} +  2 \gamma \mathcal{C}_{\rho, 1} \mu_1^{1-q} \mu_{2}^{1-p} \left(\frac{\widetilde{C}_{\Phi}}{a_{1}}\|\omega\|_{L^{1}(W)}\right)^{1-pq} \left(\frac{\theta}{\alpha}\right)^{pq} \right) \nonumber \\ \label{ThmNMvalue}
\end{eqnarray}
 with $\widetilde{C}_{\Phi}$ as in Lemma \ref{lemCphi}, the sampling inequality
 \begin{equation*}
 \widetilde{A}_{\frac{\theta}{\alpha}, \gamma} \|f\|_{L^{p,q}(\widetilde{K})} \leq  \left \|\{(f*\omega)(u_{j}, v_{k})\}_{\substack{j=1,2,\dots, n;\\ k=1,2,\dots, m}} \right\|_{\ell^{p,q}} \leq \widetilde{B}_{\frac{\theta}{\alpha}, \gamma} \|f\|_{L^{p,q}(\widetilde{K})}
 \end{equation*}
 holds with probability at least $1-\mathcal{A}_{1} e^{-nm \beta'_{\frac{\theta}{\alpha},\gamma}} -\mathcal{A}_{2} e^{-nm \beta''_{\frac{\theta}{\alpha},\gamma}},$ for every $f \in V^{p,q}_{K,\alpha}(\Phi,\theta),$ where for $\zeta>0$,
 \begin{eqnarray*}
 \widetilde{A}_{\zeta, \gamma} &=& (1-\gamma) \mathcal{C}_{\rho, 1} \mu_1^{1-q} \mu_{2}^{1-p} \left(\frac{\widetilde{C}_{\Phi}\|\omega\|_{L^{1}(W)}}{a_{1}}\right)^{(1-pq)} \zeta^{pq} n^{\frac{1}{p}} m^{\frac{1}{q}},
 \\
\widetilde{B}_{\zeta, \gamma} &=& \mathcal{C}_{\rho, 2}\mu_1^{\frac{p-1}{p}}\mu_2^{ \frac{q-1}{q} }\|\omega\|_{L^{1}(W)}\;nm \\
&& \quad \quad \quad \quad + \gamma \mathcal{C}_{\rho, 1}\mu_1^{1-q} \mu_2^{1-p} \left(\frac{\widetilde{C}_{\Phi}\|\omega\|_{L^{1}(W)}}{a_{1}}\right)^{(1-pq)} \zeta^{pq}\;nm,
\\
\beta'_{\zeta,\gamma} &=& \dfrac{\mu_1^{1-q} \mu_{2}^{1-p}\left(\frac{\sqrt{3}}{2}\gamma \mathcal{C}_{\rho, 1} \left(\frac{a_{1}\zeta}{\widetilde{C}_{\Phi}\|\omega\|_{L^{1}(W)}}\right )^{pq}\right)^{2}}{6\mu_1^{q-1} \mu_{2}^{p-1}+\gamma \mathcal{C}_{\rho, 1} \left(\frac{a_{1}\zeta}{\widetilde{C}_{\Phi}\|\omega\|_{L^{1}(W)}}\right )^{pq}},
 \\
\beta''_{\zeta,\gamma} &=&  \dfrac{\mu_1^{1-q} \mu_{2}^{1-p}\left(\gamma \mathcal{C}_{\rho, 1} \left(\frac{a_{1}\zeta}{\widetilde{C}_{\Phi}\|\omega\|_{L^{1}(W)}}\right )^{pq} \frac{\widetilde{C}_{\Phi}}{a_{1}}\right)^{2}}{72 \sqrt{2}\left(81\mu_1^{q-1} \mu_2^{p-1}+  \gamma \mathcal{C}_{\rho, 1} \left(\frac{a_{1}\zeta}{\widetilde{C}_{\Phi}\|\omega\|_{L^{1}(W)}}\right )^{pq}\frac{\widetilde{C}_{\Phi}}{a_{1}}\right)}
 \end{eqnarray*}
and  $\mathcal{A}_{1}, \mathcal{A}_{2}$ are as in Lemma \ref{lemm3.3}.
\end{thm}

 \begin{proof}
We shall first derive a sampling inequality for functions in  $V^{p,q,*}_{K}(\Phi, \zeta).$
Let  $Y_{j,k}$  be as defined  in \eqref{RandVar}. Consider the event
 $$\displaystyle \textbf{E}= \left\{\sup_{f \in V^{p,q,*}_{K}(\Phi, \zeta)}\left|\sum_{j=1}^{n}\sum_{k=1}^{m} Y_{j,k}(f) \right| > \mathfrak{p} \right \}$$ and its complement
   %of $\textbf{E},$
   $\displaystyle{\textbf{E}^{c}= \left\{-\mathfrak{p} \leq \sum_{j=1}^{n}\sum_{k=1}^{m} Y_{j,k}(f)  \leq \mathfrak{p}, \forall f \in V^{p,q,*}_{K}(\Phi, \zeta) \right \}},$
   for the random variable
 $\displaystyle\sup_{f \in V^{p,q,*}_{K}(\Phi, \zeta)}\left|\sum_{j=1}^{n}\sum_{k=1}^{m} Y_{j,k}(f) \right|$
 and a positive constant $\mathfrak{p}.$
   Using \eqref{RandVar}, we rewrite $\textbf{E}^{c}$ as
   \begin{eqnarray*}
  && \textbf{E}^{c} = \bigg\{ nm \int_{K} \rho(u,v)|(f*\omega)(u,v)|du dv -\mathfrak{p} \leq \sum_{j=1}^{n}\sum_{k=1}^{m} |(f*\omega)(u_{j}, v_{k})| \nonumber \\
   && \qquad \qquad \qquad  \leq  nm \int_{K} \rho(u,v)|(f*\omega)(u,v)|du dv  + \mathfrak{p}, \forall f \in V^{p,q,*}_{K}(\Phi, \zeta) \bigg \}.
   \end{eqnarray*}
 Suppose we assume that    $\textbf{E}^{c}$ holds true. Let $f\in  V^{p,q,*}_{K}(\Phi, \zeta).$
   Applying the Hölder's  inequality, it is easy to see that
   \begin{equation}
   \sum_{j=1}^{n}\sum_{k=1}^{m} |(f*\omega)(u_{j}, v_{k})|  \leq m^{\frac{q-1}{q}} n^{\frac{p-1}{p}} \left \|\{(f*\omega)(u_{j}, v_{k})\}_{\substack{j=1,2,\dots, n;\\ k=1,2,\dots, m}} \right\|_{\ell^{p,q}}. \label{eq30}
   \end{equation}
   For $p,q >1$, we obtain
 \begin{equation*}
  \sum_{k=1}^{m}|(f*\omega)(u_{j}, v_{k})|^{q}  \leq  \left(\sum_{k=1}^{m}|(f*\omega)(u_{j}, v_{k})| \right )^{q} \end{equation*}
  which further implies
  \begin{eqnarray*}
 \sum_{j=1}^{n} \left(\sum_{k=1}^{m}|(f*\omega)(u_{j}, v_{k})|^{q} \right )^{\frac{p}{q}} &\leq &  \sum_{j=1}^{n} \left(\sum_{k=1}^{m}|(f*\omega)(u_{j}, v_{k})| \right )^{p}  \\
 &\leq & \left( \sum_{j=1}^{n} \sum_{k=1}^{m}|(f*\omega)(u_{j}, v_{k})| \right )^{p}.
 \end{eqnarray*}
 So,
 \begin{equation*}
 \bigg \|\{(f*\omega)(u_{j}, v_{k})\}_{\substack{j=1,2,\dots, n;\\ k=1,2,\dots, m}} \bigg \|_{\ell^{p,q}}  \leq \sum_{j=1}^{n} \sum_{k=1}^{m} \big|(f \ast \omega)(u_{j}, v_{k})\big|.
 \end{equation*}
Combining the above inequality and \eqref{eq30}, we obtain
\begin{eqnarray}
 && \left \|\{(f*\omega)(u_{j}, v_{k})\}_{\substack{j=1,2,\dots, n;\\ k=1,2,\dots, m}} \right \|_{\ell^{p,q}} \nonumber \\
&& \quad \leq  \sum_{j=1}^{n}\sum_{k=1}^{m} |(f*\omega)(u_{j}, v_{k})| \nonumber \\
&&\quad \quad \leq m^{\frac{q-1}{q}} n^{\frac{p-1}{p}}
  \left \|\{(f*\omega)(u_{j}, v_{k})\}_{\substack{j=1,2,\dots, n;\\ k=1,2,\dots, m}} \right\|_{\ell^{p,q}}. \label{eq31}
\end{eqnarray}
\vskip 1em
\noindent
 As $\textbf{E}^{c}$ holds,  the above inequality together with the assumption $(\text{A}_{4})$ give
\begin{eqnarray}
 \bigg \|\{(f*\omega)(u_{j}, v_{k})\}_{\substack{j=1,2,\dots, n;\\ k=1,2,\dots, m}} \bigg \|_{\ell^{p,q}} & \leq & \sum_{j=1}^{n} \sum_{k=1}^{m} \big|(f \ast \omega)(u_{j}, v_{k})\big| \nonumber \\
& \leq & nm  \int_{K}\rho(u,v)\big|(f \ast \omega)(u, v)\big| du dv + \mathfrak{p} \nonumber \\
& \leq & nm \mathcal{C}_{\rho, 2} \|f*\omega\|_{L^{1}(K)} + \mathfrak{p}. \label{eq32}
\end{eqnarray}
Now, by applying  Hölder's  inequality twice and subsequently \eqref{mixYoung1}, it follows that
\begin{eqnarray}
 \|f*\omega\|_{L^{1}(K)} &=& \int_{K_{1}} \int_{K_{2}} |(f*\omega)(u,v)|du dv  \nonumber \\
&\leq& \mu_2^{\frac{q-1}{q}} \int_{K_{1}}\bigg(\int_{K_{2}}|(f*\omega)(u,v)|^{q} dv \bigg)^{\frac{1}{q}}du   \nonumber \\
& \leq & \mu_2^{\frac{q-1}{q}} \mu_1^{\frac{p-1}{p}} \|f*\omega\|_{L^{p,q}(K)} \nonumber \\
&\leq & \mu_2^{ \frac{q-1}{q} } \mu_1^{\frac{p-1}{p} } \|\omega\|_{L^{1}(W)}.\label{eq_1norm_conv}
\end{eqnarray}
Thus, from  \eqref{eq32}, we obtain
\begin{equation}
 \left \|\{(f*\omega)(u_{j}, v_{k})\}_{\substack{j=1,2,\dots, n;\\ k=1,2,\dots, m}} \right \|_{\ell^{p,q}} \leq nm \mathcal{C}_{\rho, 2} \mu_1^{\frac{p-1}{p} } \mu_2^{ \frac{q-1}{q}}  \|\omega\|_{L^{1}(W)} + \mathfrak{p}. \label{eq34}
\end{equation}
On the other hand, using the upper inequality of \eqref{eq31}, we get
\begin{eqnarray}
&& \bigg \|\{(f*\omega)(u_{j}, v_{k})\}_{\substack{j=1,2,\dots, n;\\ k=1,2,\dots, m}} \bigg \|_{\ell^{p,q}} \nonumber \\
&& \quad \geq n^{\frac{1-p}{p}} m^{\frac{1-q}{q}}  \sum_{j=1}^{n} \sum_{k=1}^{m} \big|(f \ast \omega)(u_{j}, v_{k})\big| \nonumber \\
&& \quad \geq n^{\frac{1-p}{p}} m^{\frac{1-q}{q}} \left(nm  \int_{K}\rho(u,v)\big|(f \ast \omega)(u, v)\big| du dv - \mathfrak{p} \right) \nonumber \\
&& \quad \geq n^{\frac{1-p}{p}} m^{\frac{1-q}{q}} \left(nm \mathcal{C}_{\rho,1} \int_{K}\big|(f \ast \omega)(u, v)\big| du dv - \mathfrak{p} \right). \label{eq35}
\end{eqnarray}
In order to estimate  $\|f*\omega\|_{L^{1}(K)},$ we make  use of the  Hölder's  inequality and obtain
\begin{eqnarray*}
 \|f*\omega\|_{L^{p,q}(K)} & = & \left(\int_{K_{1}} \left(\int_{K_{2}} |(f*\omega)(u,v)|^{q} dv\right)^{\frac{p}{q}}du \right)^{\frac{1}{p}}\nonumber \\
&  \leq &\mu_1^{\frac{q-1}{pq} } \bigg ( \int_{K_{1}} \bigg( \int_{K_{2}} |(f* \omega)(u,v)|^{q}dv  \bigg)^{p}du \bigg)^{\frac{1}{pq}} \nonumber  \\
&  \leq & \mu_1^{\frac{q-1}{pq}} \|f*\omega\|_{L^{ \infty}(K)}^{\frac{q-1}{q} }  \bigg ( \int_{K_{1}} \bigg( \int_{K_{2}} |(f* \omega)(u,v)|dv \bigg)^{p} du   \bigg)^{\frac{1}{pq}},
\end{eqnarray*}
which in turn, by  Minkowski's integral inequality, leads to
\begin{eqnarray*}
 \|f*\omega\|_{L^{p,q}(K)} &\leq &\mu_1^{\frac{q-1}{pq} } \|f*\omega\|_{L^{ \infty}(K)}^{\frac{q-1}{q} }  \bigg( \int_{K_{2}} \bigg( \int_{K_{1}} |(f* \omega)(u,v)|^{p}du  \bigg)^{\frac{1}{p}}dv  \bigg)^{\frac{1}{q}}  \\
&\leq &\mu_1^{\frac{q-1}{pq}} \|f*\omega\|_{L^{\infty}(K)}^{1-\frac{1}{pq}}  \bigg( \int_{K_{2}} \bigg( \int_{K_{1}} |(f* \omega)(u,v)|du  \bigg)^{\frac{1}{p}}dv  \bigg)^{\frac{1}{q}}.
\end{eqnarray*}
Once again using the Hölder's  inequality, we get
\begin{equation*}
 \|f*\omega\|_{L^{p,q}(K)}  \leq  \mu_1^{\frac{q-1}{pq}}\mu_2^{\frac{p-1}{pq}} \|f*\omega\|_{L^{\infty}(K)}^{1-\frac{1}{pq}}  \|f*\omega\|_{L^{1}(K)}^{\frac{1}{pq}}.
\end{equation*}
Furthermore, by  \eqref{mixYoung2} and  Lemma  \ref{lemCphi}, we obtain
\begin{equation*}
 \|f*\omega\|_{L^{p,q}(K)}  \leq  \mu_1^{\frac{q-1}{pq}} \mu_2^{\frac{p-1}{pq}} \left(\frac{\widetilde{C}_{\Phi}}{a_{1}} \|\omega\|_{L^{1}(W)} \right)^{1- \frac{1}{pq} } \|f*\omega\|_{L^{1}(K)}^{\frac{1}{pq}}.
\end{equation*}
Thus,
\begin{eqnarray*}
 \|f*\omega\|_{L^{1}(K)}
&\geq &\mu_1^{1-q} \mu_2^{1-p} \left(\frac{\widetilde{C}_{\Phi}}{a_{1}} \|\omega\|_{L^{1}(W)} \right)^{1-pq} \|f*\omega\|_{L^{p,q}(K)}^{pq}\\
&\geq &  \mu_1^{1-q} \mu_2^{1-p} \left(\frac{\widetilde{C}_{\Phi}}{a_{1}} \|\omega\|_{L^{1}(W)} \right)^{1-pq} \zeta^{pq},
\end{eqnarray*}
as  $f\in V^{p,q,*}_{K}(\Phi,\zeta).$
Therefore, from \eqref{eq35}, we get
\begin{eqnarray*}
&& \bigg \|\{(f*\omega)(u_{j}, v_{k})\}_{\substack{j=1,2,\dots, n;\\ k=1,2,\dots, m}} \bigg \|_{\ell^{p,q}}  \\
&& \geq n^{\frac{1-p}{p}} m^{\frac{1-q}{q}} \left(nm  \mathcal{C}_{\rho,1}\mu_1^{1-q} \mu_2^{1-p} \left(\frac{\widetilde{C}_{\Phi}}{a_{1}}  \|\omega\|_{L^{1}(W)}\right) ^{1-pq} \zeta^{pq}- \mathfrak{p} \right).
\end{eqnarray*}
In the presence of the above inequality and \eqref{eq34}, it is easy to observe that the event $\textbf{E}^{c} \subset \widetilde{\textbf{E}},$ where
\begin{eqnarray*}
&& \widetilde{\textbf{E}} = \bigg \{
n^{\frac{1-p}{p}} m^{\frac{1-q}{q}} \left(nm  \mathcal{C}_{\rho,1}\mu_1^{1-q} \mu_2^{1-p} \left(\frac{\widetilde{C}_{\Phi}}{a_{1}}  \|\omega\|_{L^{1}(W)}\right) ^{1-pq} \zeta^{pq}- \mathfrak{p} \right)\\
&& \quad \quad \quad  \leq \left \|\{(f*\omega)(u_{j}, v_{k})\}_{\substack{j=1,2,\dots, n;\\ k=1,2,\dots, m}} \right \|_{\ell^{p,q}} \\
&& \quad \quad \quad \quad \quad  \leq
 nm \mathcal{C}_{\rho, 2} \mu_1^{\frac{p-1}{p} } \mu_2^{ \frac{q-1}{q}}  \|\omega\|_{L^{1}(W)} + \mathfrak{p},
%nm \frac{\mathcal{C}_{\rho, 2} \|\omega\|_{L^{1}(K)}}{ \mu_1)^{\left(\frac{1-p}{p} %\right)} (\textbf{m}(K_{2}))^{ \left(\frac{1-q}{q} \right)}}  + \mathfrak{p}  , \\
\ \ \forall \ f\in V^{p,q,*}_{K}(\Phi, \zeta) \bigg \}.
\end{eqnarray*}
 The choice of
 $$\mathfrak{p}=nm  \gamma \mathcal{C}_{\rho,1}\mu_1^{1-q} \mu_2^{1-p}\left(\frac{\widetilde{C}_{\Phi}}{a_{1}} \|\omega\|_{L^{1}(W)} \right)^{1-pq} \zeta^{pq}$$
in  the inequalities present  in the event $\widetilde{\textbf{E}}$
results in the sampling inequality
\begin{equation}
 \widetilde{A}_{\zeta, \gamma}  \leq  \left \|\{(f*\omega)(u_{j}, v_{k})\}_{\substack{j=1,2,\dots, n;\\ k=1,2,\dots, m}} \right\|_{\ell^{p,q}} \leq \widetilde{B}_{\zeta, \gamma}\quad\forall\,f\in V^{p,q,*}_{K}(\Phi, \zeta),  \label{eqnnew}
 \end{equation}
 which holds with probability atleast  $1-Prob(\textbf{E}),$ with  $\mathfrak{p}$ in $\textbf{E}$  as chosen above.

\vskip 1em
Now, for analysing  the probability, suppose
\begin{equation}
nm > \frac{54 N_{1} \|\omega\|_{L^{1}(W)}}{N_{2}^{2}} (162 \|\omega\|_{L^{1}(W)}+2N_{2}), \label{eq36}
\end{equation}
where
\begin{eqnarray*}
N_{1} &=& 2\sqrt{2}(\ln 2)d \ \ \ \text{and} \\
N_{2} &=&  \gamma \mathcal{C}_{\rho, 1} \mu_1^{1-q} \mu_{2}^{1-p} \left(\frac{\widetilde{C}_{\Phi}}{a_{1}}\|\omega\|_{L^{1}(W)}\right)^{1-pq} \zeta^{pq} .
\end{eqnarray*}
Then, the  constant  $\mathfrak{p}$ satisfies the hypothesis of Lemma \ref{lemm3.3}, which we shall apply in order to estimate $Prob(\textbf{E})$. More precisely,  \eqref{eq36} holds
\begin{eqnarray*}
&\text{if and only if} & N_{2}^{2}(nm)^{2} > \left(108 N_{1} N_{2} \|\omega\|_{L^{1}(W)} + \frac{3}{N_{1}} \left(54 N_{1} \|\omega\|_{L^{1}(W)} \right)^{2} \right) nm \\
 &\text{if and only if}& \left(N_{2} nm - 54 N_{1} \|\omega\|_{L^{1}(W)} \right)^{2} > \left(54 N_{1} \|\omega\|_{L^{1}(W)} \right)^{2} \left(1+\frac{3nm}{N_{1}} \right) \\
 &\text{if and only if}& N_{2} nm > 54 N_{1} \|\omega\|_{L^{1}(W)} \left (1 + \left(1 + \frac{3nm}{N_{1}}\right)^{\frac{1}{2}} \right) \\
 &\text{if and only if}& \mathfrak{p} > 108  \sqrt{2} (\ln 2) d \left(1+ \bigg(1+\dfrac{3nm}{2  \sqrt{2} (\ln 2) d} \bigg)^{\frac{1}{2}} \right )\|\omega\|_{L^{1}(W)}.
\end{eqnarray*}
Hence, for $f \in V_{K}^{p,q,*}(\Phi, \zeta),$ the sampling inequality \eqref{eqnnew} holds with probability atleast
$ 1-\mathcal{A}_{1} \exp \left(-nm \beta'_{\zeta,\gamma}\right) - \mathcal{A}_{2} \exp \left(-nm\beta''_{\zeta,\gamma}\right),$
where $\mathcal{A}_{1} , \mathcal{A}_{2}$ are as in Lemma \ref{lemm3.3} and  $\beta'_{\zeta,\gamma},\beta''_{\zeta,\gamma}$ are as in the  hypothesis.

\vskip 1em

 Now, for $g\in V^{p,q}_{K,\alpha}(\Phi,\theta),$ we have  $ \widetilde{g}:=\frac{g}{\|g\|_{L^{p,q}(\widetilde{K})}}\in V^{p,q,*}_{K}(\Phi,\frac{\theta}{\alpha}).$
 Further, by the condition  on $n$ and $m$ given by \eqref{ThmNMvalue}, we observe that  \eqref{eq36} holds with $\zeta=\frac{\theta}{\alpha} $ in $N_{2}.$
 So by \eqref{eqnnew}, the sampling inequality
 \begin{equation}
 \widetilde{A}_{\frac{\theta}{\alpha}, \gamma} \|g\|_{L^{p,q}(\widetilde{K})} \leq  \left \|\{(g*\omega)(u_{j}, v_{k})\}_{\substack{j=1,2,\dots, n;\\ k=1,2,\dots, m}} \right\|_{\ell^{p,q}} \leq \widetilde{B}_{\frac{\theta}{\alpha}, \gamma} \|g\|_{L^{p,q}(\widetilde{K})} \nonumber \\
 \end{equation}
  holds with probability  at least $1-\mathcal{A}_{1} e^{-nm \beta'_{\frac{\theta}{\alpha},\gamma}} -\mathcal{A}_{2} e^{-nm \beta''_{\frac{\theta}{\alpha},\gamma}},$ thereby proving the theorem.
\hfill{$\Box$}
 \end{proof}
%\begin{thm}\textbf{Sampling inequality}
%\end{thm}

%\section{Reconstruction algorithm} \label{sec4}
%Finally, the reconstruction algorithm will be introduced for the functions in the subset of finite dimensional shift invariant subspace of %$L^{p,q}(G_{1}\times G_{2}).$\\

%\begin{thm}
%\textbf{Reconstruction formula }
%\end{thm}
The sampling inequality for the subset $V_{K}(\Phi,\omega,\mu)$ of $V_K(\Phi)$ is discussed in the theorem below.
\begin{thm}\label{Sampl_Ineq_2}
Suppose $\omega,\,\Phi,\,\rho$ and the i.i.d random variables $\{(u_j,v_k)\}_{j,k\in\mathbb{N}}$ over $K$ be as in  Theorem \ref{Sampl_Ineq_1}. Let $0<\mu\leq 1$ and $0<\eta<\mu \mathcal{C}_{\rho,1}$. Taking the number of samples $nm$ more than $\tfrac{108\sqrt{2}(\ln 2)d}{\eta}\left(2+\frac{162}{\eta}\right)$, the following sampling inequality holds for a function $f\in V_{K}(\Phi,\omega,\mu)$\emph{:}
\begin{align*}
&nm\|\omega\|_{L^1(W)}\left(\mu\mathcal{C}_{\rho,1}-\eta\right)\|f\|_{L^{p,q}(\widetilde{K})}\\
&\quad\leq\sum_{j=1}^{n}\sum_{k=1}^{m}|(f\ast\omega)(u_j,v_k)|\\
&\quad\quad\leq nm\|\omega\|_{L^1(W)}\left(\mathcal{C}_{\rho,2}\mu_{1}^{\frac{p-1}{p}}\mu_{2}^{\frac{q-1}{q}}+\eta\right)\|f\|_{L^{p,q}(\widetilde{K})}
\end{align*}
with probability atleast
\[1-\mathcal{A}_1exp\left(-nm\frac{3a_1^2\eta^2}{4\widetilde{C}_{\Phi}(6\widetilde{C}_{\Phi}+\eta a_1)}\right)-\mathcal{A}_2exp\left(-nm\frac{\eta^2}{72\sqrt{2}(81+\eta)}\right),\]
where $\mathcal{A}_1, \mathcal{A}_2$ and $\widetilde{C}_{\Phi}$ are as in Lemma \ref{lemm3.3}.
\end{thm}

\begin{proof}
Consider $f\in V_{K}(\Phi,\omega,\mu)$ and $\tilde{f}=\frac{f}{\|f\|_{L^{p,q}(\widetilde{K})}}$. If we assume that the random variables $Y_{j,k}(\tilde{f})$, given by \eqref{RandVar}, satisfy
$$\left|\sum_{j=1}^{n}\sum_{k=1}^{m}Y_{j,k}(\tilde{f})\right|\leq nm\eta\|\omega\|_{L^{1}(W)}$$ with $n,m\in\mathbb{N}$ as in the hypothesis, then we have
\begin{align*}
&nm\left(\int_{K}\rho(u,v)|(\tilde{f}\ast\omega)(u,v)|dudv-\eta\|\omega\|_{L^{1}(W)}\right)\\
&\quad\leq\sum_{j=1}^{n}\sum_{k=1}^{m}|(\tilde{f}\ast\omega)(u_j,v_k)|\\
&\quad\quad\leq nm\left(\int_{K}\rho(u,v)|(\tilde{f}\ast\omega)(u,v)|dudv+\eta\|\omega\|_{L^{1}(W)}\right).
\end{align*}
Further, by the assumption $(A_4)$ on $\rho$, we obtain
\begin{align*}
&nm\left(\mathcal{C}_{\rho,1}\|\tilde{f}\ast\omega\|_{L^{1}(K)}-\eta\|\omega\|_{L^{1}(W)}\right)\\
&\quad\leq\sum_{j=1}^{n}\sum_{k=1}^{m}|(\tilde{f}\ast\omega)(u_j,v_k)|\\
&\quad\quad\leq nm\left(\mathcal{C}_{\rho,2}\|\tilde{f}\ast\omega\|_{L^{1}(K)}+\eta\|\omega\|_{L^{1}(W)}\right).
\end{align*}
Now, by the definition of $V_{K}(\Phi,\omega,\mu)$ and \eqref{eq_1norm_conv}, we get
\begin{align*}
\mu\|\omega\|_{L^1(W)}\leq\|\tilde{f}*\omega\|_{L^{1}(K)}\leq \mu_{2}^{\frac{q-1}{q}} \mu_{1}^{\frac{p-1}{p}} \|\omega\|_{L^{1}(W)}
\end{align*}
and so
\begin{align*}
&nm\|\omega\|_{L^1(W)}\left(\mu\mathcal{C}_{\rho,1}-\eta\right)\|f\|_{L^{p,q}(\widetilde{K})}\\
&\quad\leq\sum_{j=1}^{n}\sum_{k=1}^{m}|(f\ast\omega)(u_j,v_k)|\\
&\quad\quad\leq nm\|\omega\|_{L^1(W)}\left(\mathcal{C}_{\rho,2}\mu_{1}^{\frac{p-1}{p}}\mu_{2}^{\frac{q-1}{q}}+\eta\right)\|f\|_{L^{p,q}(\widetilde{K})}.
\end{align*}
Thus, the required sampling inequality holds with probability atleast
\begin{align*}
&Prob\left(\sup_{\substack{f\in V_{K}(\Phi,\omega,\mu)\\\|f\|_{L^{p,q}(\widetilde{K})}=1}}\left|\sum_{j=1}^{n}\sum_{k=1}^{m}Y_{j,k}(f)\right|\leq nm\eta\|\omega\|_{L^{1}(W)}\right)\\
%&\quad=1-Prob\left(\sup_{\substack{f\in V^{p,q}_{N,\psi}(\Phi,\mu,C_K)\\\|f\|_{L^{p,q}%(\mathbb{R}\times\mathbb{R}^d)}=1}}\left|\sum_{j=1}^{n}\sum_{k=1}^{m}Y_{j,k}(f)\right|%\geq nm\eta\|\psi\|_{L^{1,1}(C_K)}\right)\\
&\quad\geq 1-Prob\left(\sup_{f\in V^{\ast}_{K}(\Phi)}\left|\sum_{j=1}^{n}\sum_{k=1}^{m}Y_{j,k}(f)\right|\geq nm\eta\|\omega\|_{L^{1}(W)}\right).
\end{align*}
As $nm>\frac{108\sqrt{2}(\ln 2)d}{\eta}\left(2+\frac{162}{\eta}\right)$, it can be shown that
\begin{align*}
nm\eta\|\omega\|_{L^{1}(W)}&>108\sqrt{2}(\ln 2)d\left(1+\left(1+\frac{3nm}{2\sqrt{2}(\ln 2)d}\right)^{\frac{1}{2}}\right)\|\omega\|_{L^{1}(W)}
\end{align*}
and hence, by applying Lemma \ref{lemm3.3} with $\mathfrak{p}=nm\eta\|\omega\|_{L^{1}(W)}$, the probability can be estimated, thereby proving the result.
\hfill{$\Box$}
\end{proof}

\section{Reconstruction formulae} \label{ReconFor}
The reconstruction formulae for the functions in the space  $V_{K}(\Phi)$, under certain conditions, are discussed in the following two theorems.
%and $V_{K}(\Phi, \omega, \mu)$ shall be discussed in this section.

\begin{thm}\label{Recon1}
Let $ \omega, \Phi$ and $\rho$  satisfy the assumptions $(\text{A}_{1})- (\text{A}_{4})$ and $\{(u_{j}, v_{k})\}_{j,k \in \mathbb{N}}$  be a  sequence of i.i.d. random variables over $K$ with $\rho$ as the probability density function.
Suppose there exists a positive constant $\beta$ such that
\begin{equation}
\left\|\sum_{s,t} \textbf{c}(s,t)^{T} (\Phi * \omega)(\cdot - x_{s}, \cdot - y_{t})\right\|_{L^{p,q} (K)} \geq \beta\|\textbf{c}\|_{\ell^{p,q}},   \label{4eq1}
\end{equation}
 for all $\textbf{c} \in \left(\ell^{p,q}\right)^{r} .$
 Then, for any $\gamma \in (0,1)$ and $n,m\in {\mathbb{N}}$ satisfying \eqref{ThmNMvalue}, there exist  functions $ h_{j,k}$  such that
$$f(u, v)= \sum_{j=1}^{n} \sum_{k=1}^{m}(f*\omega)(u_{j}, v_{k}) h_{j,k}(u,v),\quad (u,v)\in \widetilde{K},$$  for all $f \in V_{K}(\Phi)$ with probability at least $1-\mathcal{A}_{1} e^{-nm \beta'_{\frac{\beta}{a_2},\gamma}} -\mathcal{A}_{2} e^{-nm \beta''_{\frac{\beta}{a_2},\gamma}},$
 where $ \mathcal{A}_{1},\mathcal{A}_{2},  \beta'_{\frac{\beta}{a_2},\gamma}$ and  $\beta''_{\frac{\beta}{a_2},\gamma}$ are  as in Theorem \ref{Sampl_Ineq_1}.
\end{thm}

\begin{proof}
Let $f \in V_{K}(\Phi).$ Then, without loss of generality, we may write $f$ as
\begin{equation}
f = \sum_{i=1}^{r}\sum_{s=1}^{s_0} \sum_{t=1}^{t_{0}} c_{i}(s, t) \phi_{i}(\cdot - x_{s}, \cdot - y_{t}) \text{ on } \widetilde{K}, \label{4eq2}
\end{equation}
where $s_{0}$ and $t_{0}$ are chosen, by appropriately reindexing, such that  $\Phi(\cdot - x_{s}, \cdot - y_{t})\not\equiv 0 \text{ on } \widetilde{K}$ for $1\le s\le s_{0},1\le t\le t_{0}$ and it is identically zero   on  $\widetilde{K}$ otherwise.
So,
\begin{equation}
(f*\omega)(u_{j}, v_{k}) = \sum_{i=1}^{r}\sum_{s=1}^{s_0} \sum_{t=1}^{t_{0}} c_{i}(s, t) (\phi_{i} * \omega )(u_{j}- x_{s}, v_{k}-y_{t})  \label{4eq3}
\end{equation}
for $ 1 \leq j \leq n $ and $ 1 \leq k \leq m .$
We shall  rewrite these equations  in terms of  matrices.
Let us consider $S_{jk} := (f * \omega)(u_{j}, v_{k}) \text{ for } 1 \leq j \leq n, 1 \leq k \leq m $ and the $r$-tuple  $$\ M_{jk}(s, t):= \big( m_{1, s, t}(u_{j}, v_{k}), m_{2, s, t}(u_{j}, v_{k}), \dots, m_{r, s, t}(u_{j}, v_{k}) \big ),$$  where
$ m_{i,s, t}(u_{j}, v_{k}) := (\phi_{i} * \omega)(u_{j}-x_{s}, v_{k} - y_{t}), \; 1 \leq i \leq r. $
Further, let
\begin{eqnarray*}
 M_{jk}(s)  &:=& \left(M_{jk}(s, 1), M_{jk}(s, 2),\ldots, M_{jk}(s, t_{0})  \right), \; \text{ for fixed } j,k,s\\
\mbox{and } M_{jk} &:=& \left(M_{jk}(1), M_{jk}(2),\ldots,  M_{jk}(s_{0}) \right),  \; \text{ for fixed } j,k.
\end{eqnarray*}
 Similarly, we assume
\begin{equation*}
    \begin{array}{rll}
        \textbf{c}(s,t) &:=& \big (c_{1}(s,t), c_{2}(s,t), \cdots, c_{r}(s,t) \big )^{T}, \\
       \textbf{c}(s)    &:=& \left(\textbf{c}(s,1), \textbf{c}(s,2),\ldots,\textbf{c}(s,t_{0})\right)^{T},  \\
      \textbf{c} & :=&\left(\textbf{c}(1), \textbf{c}(2), \ldots,\textbf{c}(s_{0})\right)^{T}.
    \end{array}
 \end{equation*}

 \noindent
Therefore,
\begin{eqnarray*}
S_{jk} &=& \sum_{i=1}^{r} \sum_{s=1}^{s_0} \sum_{t=1}^{t_{0}} c_{i}(s, t) m_{i, s, t}(u_{j}, v_{k})  \\
&=& \sum_{s=1}^{s_0} \sum_{t=1}^{t_{0}} M_{jk}( s, t)\textbf{c}(s,t)  \\
&=& \sum_{s=1}^{s_0} M_{jk}(s)\textbf{c}(s)  \\
&=& M_{jk}\textbf{c}  \\
&=& (M \textbf{c})_{jk}.
\end{eqnarray*}
%where the matrix $M_{jk}$ is of order $nm \times s_{0}.$
  Thus, the system of equations \eqref{4eq3} can be represented as
\begin{equation}
M\textbf{c}=S, \label{4eq4}
\end{equation}
where $M,\textbf{c}$ and $S$ are   $mn\times rs_{0}t_{0},$ $ rs_{0}t_{0}\times 1$ and $mn\times 1$ matrices respectively.
\par In order to prove the injectivity of $M,$ we consider
$\widetilde{f}:=\frac{\alpha f}{\|f\|_{L^{p,q}(\widetilde{K})}}.$ Then, using \eqref{4eq1} and \eqref{asmptn1}, we obtain
%By Theorem \ref{thm3.5} and  using \eqref{4eq1}, we get
 \begin{equation*}
\|\widetilde{f}*\omega \|_{L^{p,q}(K)} \geq \beta \|\widetilde{\textbf{c}}\|_{\ell^{p,q}} \geq \tfrac{\beta}{a_{2}} \|\widetilde{f}\|_{L^{p,q}(\widetilde{K})} = \frac{\alpha\beta}{a_{2}},
\end{equation*}
where $\widetilde{\textbf{c}}$ is the coefficient vector corresponding to $\widetilde{f}.$ Thus,  $\widetilde{f} \in V_{K,\alpha}^{p,q}(\Phi,  \tfrac{\alpha\beta}{a_{2}}).$
 Now, by Theorem \ref{Sampl_Ineq_1}, the inequality
\begin{equation}
 \widetilde{A}_{\frac{\beta}{a_{2}}, \gamma} \|f\|_{L^{p,q}(\widetilde{K})}  \leq \bigg \| \left \{(f*\omega)(u_{j}, v_{k}) \right \}_{\substack{j=1,2,\dots, n;\\ k=1,2,\dots, m}} \bigg \|_{\ell^{p,q}} \label{eq4new1}
\end{equation}
holds with probability atleast $1-\mathcal{A}_{1} e^{-nm \beta'_{\frac{\beta}{a_{2}},\gamma}} -\mathcal{A}_{2} e^{-nm \beta''_{\frac{\beta}{a_{2}},\gamma}}$.  Making use of the matrix equation $M\textbf{c}=S$  in  \eqref{eq4new1}, we have
\begin{equation*}
\|M \textbf{c}\|_{\ell^{p,q}} \geq \widetilde{A}_{\frac{\beta}{a_{2}}, \gamma} \|f\|_{L^{p,q}(\widetilde{K})} \geq \widetilde{A}_{\frac{\beta}{a_{2}}, \gamma} a_{1} \|\textbf{c}\|_{\ell^{p,q}},
\end{equation*}
by  \eqref{asmptn1}. As the above inequality holds for all $\textbf{c} \in {\mathbb{C}}^{rs_{0}t_{0}},$ it follows that $M$ is injective.
\par
Further, let $M^{*}$ be the conjugate transpose of $M.$ Then,  the operator $M^{*}M$ is also injective and therefore invertible. Using \eqref{4eq4} and the invertibility of $M^{*}M$, we get  $\textbf{c} = (M^{*} M)^{-1}M^{*}S = (\widetilde{M})^{T}S,$
  where $\widetilde{M}$ denotes the matrix $\overline{M}(M^{T} \overline{M})^{-1}.$
 We let   $\widetilde{M}_{jk}, \widetilde{M}_{jk}(s)$ and $\widetilde{M}_{jk}(s,t)$ denote  the tuples associated with $\widetilde{M},$ similar to  those of $M.$  Also,  we take  $\left( \widetilde{M}_{jk}(s,t) \right)_{i}$ to represent the $i^{th}$ coordinate of  $\widetilde{M}_{jk}(s,t).$ Thus,  the unknown coefficients can be written as
 %\begin{eqnarray*}
 % \textbf{c}(s) &=& \sum_{j,k} \widetilde{M}_{jk}(s)S_{jk} \; \text{ for } x_{s} \in X \\
%\text{and } \quad \textbf{c}(s,t) &=& \displaystyle\sum_{j,k} \widetilde{M}_{jk}(s,t)S_{jk} \; \text{ for } y_{t} \in Y,
 %\end{eqnarray*}
% where
$$\text{c}_{i}(s,t)= \sum_{j,k} \big ( \widetilde{M}_{jk}(s,t) \big )_{i}S_{jk}, \mbox{ for } 1\le i\le r, 1\le s\le s_{0}  \mbox{ and  } 1\le t\le t_{0}.$$
%denotes the $i^{th}$ coordinate.
Therefore, \eqref{4eq2} leads to
\begin{eqnarray*}
f (u,v)&=& \sum_{j,k} S_{jk}  \sum_{s=1}^{s_{0}}  \sum_{t=1}^{t_{0}} \sum_{i=1}^{r} \big ( \widetilde{M}_{jk}(s,t) \big )_{i}  \phi_{i}(u - x_{s}, v - y_{t})\\
&=&\sum_{j,k}(f*\omega)(u_{j}, v_{k})h_{j,k}(u,v), \ \ (u,v) \in \widetilde{K},
\end{eqnarray*}
where $\displaystyle{h_{j,k}:= \sum_{s=1}^{s_{0}}  \sum_{t=1}^{t_{0}} \sum_{i=1}^{r} \big ( \widetilde{M}_{jk}(s,t) \big )_{i}  \phi_{i}(\cdot - x_{s}, \cdot - y_{t}).}$  Clearly, this reconstruction formula holds with probability
 at least $1-\mathcal{A}_{1} e^{-nm \beta'_{\frac{\beta}{a_2},\gamma}} -\mathcal{A}_{2} e^{-nm \beta''_{\frac{\beta}{a_2},\gamma}},$ thereby proving the result.
 \hfill{$\Box$}
\end{proof}

\begin{thm}
Let $ \omega, \Phi, \rho$ and   $\{(u_{j}, v_{k})\}_{j,k \in \mathbb{N}}$ be  as in Theorem \ref{Recon1}. Suppose there exists a constant $0<\widetilde{\mu}\leq 1,$ such that $\widetilde{\mu}\|\omega\|_{L^1(W)}\|f\|_{L^{p,q}(\widetilde{K})}\leq\|f\ast\omega\|_{L^{1}(K)},$ for all $f\in V_K(\Phi).$ Then, for $ 0<\eta<\widetilde{\mu}\mathcal{C}_{\rho,1},$ and $n,m\in\mathbb{N}$ such that $nm>\tfrac{108\sqrt{2}(\ln 2)d}{\eta}\left(2+\frac{162}{\eta}\right)$, there exists a sequence of functions $\{h_{j,k}:1\leq j\leq n,\,1\leq k\leq m\}$ such that every $f\in V_K(\Phi)$ can be reconstructed from its average sample values as
$$f(u, v)= \sum_{j=1}^{n} \sum_{k=1}^{m}(f*\omega)(u_{j}, v_{k}) h_{j,k}(u,v),\quad (u,v)\in \widetilde{K},$$
with probability atleast
\[1-\mathcal{A}_1exp\left(-nm\frac{3a_1^2\eta^2}{4\widetilde{C}_{\Phi}(6\widetilde{C}_{\Phi}+\eta a_1)}\right)-\mathcal{A}_2exp\left(-nm\frac{\eta^2}{72\sqrt{2}(81+\eta)}\right),\]
where $\mathcal{A}_1, \mathcal{A}_2$ and $\widetilde{C}_{\Phi}$ are as in Lemma \ref{lemm3.3}.
\end{thm}

\begin{proof}
Let $f\in V_K(\Phi).$ Then, $f$ may be written as in \eqref{4eq2} and from \eqref{asmptn1} we have
$$ a_{1}\|\textbf{c}\|_{\ell^{p,q}} \leq \|f\|_{L^{p,q}(\widetilde{K})} \leq a_{2} \|\textbf{c}\|_{\ell^{p,q}}.$$ Also $f\in V_{K}(\Phi,\omega,\widetilde{\mu})$, by hypothesis. Therefore, applying Theorem \ref{Sampl_Ineq_2} and the above inequality, we obtain
\begin{align*}
&nm\|\omega\|_{L^1(W)}\left(\widetilde{\mu}\mathcal{C}_{\rho,1}-\eta\right)a_{1}\|\textbf{c}\|_{\ell^{p,q}}\\
&\quad\leq\sum_{j=1}^{n}\sum_{k=1}^{m}|(f\ast\omega)(u_j,v_k)|\\
&\quad\quad\leq nm\|\omega\|_{L^1(W)}\left(\mathcal{C}_{\rho,2}\mu_{1}^{\frac{p-1}{p}}\mu_{2}^{\frac{q-1}{q}}+\eta\right)a_{2} \|\textbf{c}\|_{\ell^{p,q}},
\end{align*}
which is probabilistic. Making use of the matrices $M, \textbf{c}$ and $S$, defined as in the proof of Theorem \ref{Recon1}, we get $a'_1 \|\textbf{c}\|_{\ell^{p,q}}\leq \|M\textbf{c}\|_{\ell^1}\leq a'_2 \|\textbf{c}\|_{\ell^{p,q}}$ for some positive constants $a'_1$ and $a'_2$. In fact, this is true for all such $\textbf{c}$, thereby rendering $M$ injective. The proof now continues as in that of Theorem \ref{Recon1}, but with the probability as in Theorem \ref{Sampl_Ineq_2}.
\hfill{$\Box$}
\end{proof}

\vskip 1 em
\noindent {\bf Acknowledgment:}
The second author S. Arati acknowledges the financial support of National Board for Higher Mathematics, Department of Atomic Energy(Government of India). The third author P. Devaraj acknowledges the financial support of the Department of Science and Technology(Government of India) under the research grant DST-SERB Research Grant MTR/2018/000559.

% ------------------------------------------------------------------------
\end{document}